\documentclass[11pt]{article}

\usepackage{subfigure}
\usepackage{color}
\usepackage[cmex10]{amsmath}
\usepackage{mathrsfs,amssymb,amsmath}
\usepackage{cases}
\usepackage{graphicx}
\usepackage{caption}

\usepackage{epstopdf}
\usepackage{hyperref}

\usepackage{float}
\usepackage{placeins}

\newcommand{\R}{{\mathbb R}}

\def\dref#1{(\ref{#1})}
\def\crr{\cr\noalign{\vskip-0.0mm}}

\usepackage{algorithm} 
\usepackage{algorithmic} 

\newtheorem{corollary}{Corollary}[section]
\newtheorem{remark}{Remark}[section]
\newtheorem{lemma}{Lemma}[section]

\newtheorem{proposition}{Proposition}

\newtheorem{theorem}{Theorem}[section]
\newtheorem{definition}{Definition}[section]
\newtheorem{notation}{Notation}[section]

\newenvironment{proof}{{\bf {Proof.}}}{\hfill $\square$}
\allowdisplaybreaks[4]

\numberwithin{equation}{section}

\def\dref#1{(\ref{#1})}

\def\pt{\partial}

\def\ra{\rightarrow}
\def\bs{\boldsymbol}

\def\s{\subseteq}

\def\e{\varepsilon}

\def\ol{\overline}

\def\vp{\varphi}
\def\lg{\langle}
\def\llg{\left\langle}

\def\rg{\rangle}
\def\rrg{\right\rangle}

\def\bf{\textbf}
\def\pt{\partial}

\def\Om{\Omega}
\def\la{\lambda}
\def\al{\alpha}

\def\de{\delta}

\def\De{\Delta}

\def\Ga{\Gamma}

\def\ts{\times}

\def\iy{\infty}

\def\f{\frac}

\def\se{\setminus}

\def\Lra{\Leftrightarrow}

\def\Span{{\rm{span}}}

\def\df{\mathrm d}

\def\wt{\widetilde}
\def\wh{\widehat}

\def\esssup{\operatorname*{ess\ \! sup}}

\def\hra{\hookrightarrow}

\def\mcO{\mathcal{O}}
\def\mcA{\mathcal{A}}

\def\mcM{\mathcal{M}}

	\DeclareMathOperator{\Div}{div}

	\DeclareMathOperator{\dist}{dist}

	\DeclareMathOperator{\supp}{{supp}}

	\newcommand{\N}{\mathbb N}

	\newcommand{\pxi}{\f{\pt}{\pt x^i}}
	\newcommand{\pxj}{\f{\pt}{\pt x^j}}
	\newcommand{\pxk}{\f{\pt}{\pt x^k}}

\begin{document}

\title{{\bf   {\bf  Approximation of Degenerate Hyperbolic Equations with Interior Degeneracy and Applications to Controllability}}\footnote{\small This work was carried out with the support of the
National Natural Science Foundation of China under grant  nos. 12131008 and U23B2033, and
National Key R\&D Program of China under grant no. 2024YFA1013101.}}

\author{ Dong-Hui Yang$^{a}$, Bao-Zhu Guo$^{b}$\footnote{\small
The corresponding author. Email: bzguo@is.ac.cn}
\\
$^a${\it School of Mathematics and Statistics, Central South University}\\
			{\it Changsha 410075, P.R.China}\\
$^b${\it Academy of Mathematics and Systems Science, Academia Sinica, Beijing 100190, China}}

\date{}

\maketitle{}

\begin{abstract}
      In this paper, we establish the existence of solutions for a particular class of degenerate hyperbolic equations. Following this, we approximate these degenerate equations by employing a sequence of uniformly hyperbolic equations. Notably, this specific approximation result has remained unexplored in the existing body of literature. Ultimately, we utilize this approximation framework to derive controllability results for the original degenerate hyperbolic equations, marking what could potentially be the inaugural investigation into higher-dimensional degenerate hyperbolic equations.

\vspace{0.3cm}

	
\vspace{0.3cm}
		

	\end{abstract}

\section{Introduction}\label{S1}
	
Hyperbolic equations have attracted substantial research focus in the existing body of literature; for detailed references, see \cite{Evans,Fattorini,Fursikov,Lasiecka,Lasiecka1,Lasiecka2,Rousseau,Russell,Zuazua}. Degenerate hyperbolic equations have similarly sparked significant scholarly interest; for further reading, refer to \cite{Bai,Bai1,Cannarsa1,Gueye,Han,Martinez,Yang,Zhang}.

The approximation of degenerate elliptic or parabolic equations by sequences of uniformly elliptic or parabolic equations has been thoroughly investigated in numerous studies; for relevant citations, see \cite{Cavalheiro,FF,Wu1,Yang1,Yang3,Yang4}. However, to the best of the authors' knowledge, this specific strategy has not yet been explored for degenerate hyperbolic equations.

Approximating degenerate hyperbolic equations with uniformly hyperbolic counterparts is not only inherently fascinating but also holds immense importance for a variety of mathematical applications.

Controllability represents a pivotal field in both pure and applied mathematics, and it has been exhaustively studied for uniformly partial differential equations; for illustrative examples, refer to \cite{Coron,Fattorini,Fattorini1,Fursikov,Lasiecka,Lasiecka1,Lasiecka2,Lasiecka3,Lions,Rousseau,Russell,Yao,Zuazua}. We recall that, for uniformly hyperbolic equations, controllability is equivalent to observability through the Hilbert Uniqueness Method (HUM) introduced by J.-L. Lions \cite{Lions}. Therefore, in order to study the controllability of a uniformly hyperbolic equation, it suffices to establish the corresponding observability inequality.

There is also a plethora of research dedicated to the controllability of degenerate hyperbolic equations; for specific studies, see \cite{Bai,Bai1,Cannarsa1,Fragnelli,Gueye,Martinez,Yang4,Yang}. Nevertheless, the majority of these findings are confined to one-dimensional degenerate hyperbolic equations.

The papers \cite{Cannarsa1,Fragnelli,Gueye} addressed this issue within the one-dimensional context. However, their approaches do not seem readily adaptable to higher-dimensional problems. These results were later extended in \cite{Bai,Bai1}, yet the higher-dimensional scenario remains largely unexplored.

For higher-dimensional degenerate hyperbolic equations, the works \cite{Yang4,Yang} proposed a relatively general approach, termed the cut-off method, which is applicable to a wide array of degenerate hyperbolic problems. Essentially, this method involves isolating the degenerate region and applying controls to this segment, thereby achieving controllability results for the original degenerate hyperbolic equation. 

Although the cut-off method has certain practical applications, from our viewpoint, it is not entirely satisfactory, as it requires the control to act directly on the degenerate region itself. From both physical and mathematical perspectives, it is more intuitive to consider controls that operate either on a segment of the degenerate region or on a non-degenerate portion of the domain.

	 In this study, we investigate the null controllability of the following degenerate hyperbolic equation:
	\begin{equation}\label{12.14.1}
		\begin{cases}
			\partial_{tt}\varphi-\Div(|x|^\alpha \nabla \varphi)=f, &\text{in }Q, \\
			\varphi=0, &\text{on }  \pt Q, \\
			\varphi(0)=\varphi^0, \pt_t\vp(0)=\vp^1, &\text{in }\Omega,
		\end{cases}
	\end{equation}
	where $\alpha\in (0,2)$ is a given constant, $\Omega\subset\mathbb{R}^N$ is a bounded domain  with $0\in\Om$ and  $\partial\Omega\in C^2$,   $Q=\Omega\times (0,T)$ with $T>0$ being a constant,
	and $\pt Q=\pt\Om \ts (0,T)$. The initial data are
	$\varphi^0\in H_0^1(\Omega;w)$ and $ \vp^1\in L^2(\Om)$,  and $f\in L^2(Q)$. Here and in what follows, the weighted Sobolev spaces $H_0^1(\Om;w), H^1(\Om;w)$ and $H^2(\Om;w)$ will be defined in the next section.

	Denote
	\begin{equation}\label{11.21.8}
		w=|x|^\alpha,\quad w(E)=\int_E w\df x,
	\end{equation}
	and
	\begin{equation}\label{gbz1}
		\mcA \vp=-\Div(w\nabla \vp),\; D(\mcA)=H^2(\Om;w)\cap H_0^1(\Om;w).
	\end{equation}
	 It is evident that   $w$ is an $A_{1+\f{2}{N}}$ weight (see \cite{GC,Heinonen}).  Moreover, we denote  $R_0>0$ such that $B(0,8R_0)\equiv\{x\in\R^N\colon |x|<8R_0\}\s \Om$.
	
In this paper, we first establish the existence of solutions for the degenerate hyperbolic equation \dref{11.21.8}. Subsequently, we approximate this degenerate equation by a sequence of uniformly hyperbolic equations. To the best of our knowledge, this type of approximation
 and the controllability of higher-dimensional degenerate hyperbolic equations have not been previously addressed in the existing literature. Finally, by leveraging the controllability properties of the approximating uniformly hyperbolic equations, we derive controllability results for the system \dref{11.21.8}.
 
{{
		Although approximation methods for degenerate equations have been considered in \cite{Cavalheiro,FF}, they are not sufficient for our purposes, since the approximating functions used there are merely measurable and lack the regularity required in our analysis. In contrast, the approximation introduced in \cite{Yang4} and further developed in the present paper is suitable, as it possesses second-order partial derivatives.
		
		We emphasize that choosing an appropriate approximation of the weight function $|x|^\alpha$ is crucial. In particular, the approximating weights $w_\varepsilon$ must satisfy additional structural and regularity properties, including precise control of their derivatives. For example, in the proof of Lemma \ref{08.16.L2}, we require
		 $\psi_\varepsilon - x\cdot\nabla \psi_\varepsilon\geq  \frac{3}{8\varepsilon^3}(\varepsilon^2 - |x|^2)^2 \ge 0$.

		It is well known that observability inequalities can be derived via the multiplier method. In practice, there are two main approaches: one based on classical integrations by parts (see, e.g., \cite[Chapter 2.3]{Zuazua}), and the other based on geometric arguments (see \cite{Lasiecka3,Yao}). For degenerate hyperbolic equations, the geometric approach is often more suitable. In particular, once Lemmas \ref{11.26.L1} and \ref{11.26.L2} in the following are established, the derivation of the observability inequality becomes relatively routine. However, in the degenerate setting, these two lemmas are highly nontrivial and constitute the main analytical difficulty. 
 We also note that, for certain special domains $\Omega$, the classical integration-by-parts method remains effective.

In order to derive the observability inequality for equation \eqref{12.14.1}, we begin by selecting an appropriate multiplier $H(x)$, which we take to be $H(x)=x$ (see Lemma \ref{11.26.L1}). For the uniformly hyperbolic equations, a key term arising in the multiplier method is
 \begin{equation*}
 	\iint_Q \operatorname{div}(|x|^\alpha \nabla \varphi)\,(H\cdot \nabla \varphi)\df x\df t, \mbox{ or } \iint_Q \Div(|x|^\al \nabla\vp)\left[H(\vp)\right]\df x\df t.
 \end{equation*}
 However, this term cannot be handled directly by integration by parts, since it is not clear whether $H\cdot\nabla\varphi \in H^1(\Omega;w)$ (or, $H^1(\Omega)$). To justify such an argument, one would need higher-order regularity for $\varphi$. Although we know that $\varphi \in H^2(\Omega;w)\cap H_0^1(\Omega;w)$, this does not guarantee that the mixed derivatives $\partial_{x_i x_j}\varphi$ belong to the required weighted spaces for $i,j=1,\dots,N$. Establishing this level of regularity is technically difficult.
 
 In fact, we have shown in \cite[Lemma 2.6]{Yang3} that
 \begin{equation*}
 	|x|^{3+\f{3\al}{2}}\,\frac{\partial^2 \varphi}{\partial x_i \partial x_j} \in L^2(Q),
 \end{equation*} 
 but this is still insufficient to justify a direct integration by parts in the above expression. Therefore, a direct approach appears to be {\it infeasible}, and an approximation argument is more suitable for our method.}}

We now present the main results of this paper. The first result establishes an approximation of the degenerate hyperbolic equation by a sequence of uniformly hyperbolic equations.
\begin{theorem}\label{11.20.T2}
		Let $\vp^0=\vp_\e^0\in H_0^1(\Om), \vp^1=\vp_\e^1\in L^2(\Om)$  for all $\e\in (0,R_0)$. Let $\vp$ be the solution of \eqref{12.14.1} with respect to $(\vp^0,\vp^1,f)$,  and $\vp_\e\ (0<\e<R_0)$ be the solution of \eqref{11.14.2} with respect to $(\vp^0,\vp^1,f)$. Then exists a subsequence of $\{\vp_\e\}_{\e>0}$, still denoted by itself,  such that
		\begin{equation}\label{11.21.2}
			\begin{split}
				\vp_\e
				&\ra \vp \mbox{ weakly in } L^2(0,T; H_0^1(\Om;w)), \\
				\pt_t\vp_\e
				&\ra \pt_t\vp \mbox{ weakly in } L^2(Q).
			\end{split}
		\end{equation}
		Moreover, if we have an additional assumption $\vp^0\in C_0^\iy(\Om), \vp^1\in H_0^1(\Om)$ and $f\in H^1(0,T; L^2(\Om))$, then
		\begin{equation}\label{11.21.7}
			\f{\pt \vp_\e}{\pt \nu}\ra \f{\pt \vp}{\pt \nu} \mbox{ strongly in } L^2(Q).
		\end{equation}
	\end{theorem}

	 Next, we obtain the controllability for the degenerate hyperbolic equation:
\begin{equation}\label{12.16.2*}
	\begin{cases}
		\pt_{tt}\vp+\mcA\vp=0, &\mbox{in }Q, \\
		\vp=u, &\mbox{on }\Sigma_0, \\
		\vp=0, &\mbox{on }\Sigma\setminus\Sigma_0, \\
		\vp(0)=\vp^0,  \pt_t\vp(0)=\vp^1, &\mbox{in }\Om,
	\end{cases}
\end{equation}
where $\mcA$ is defined by \dref{gbz1}, $\vp^0\in L^2(\Om)$,  $\vp^1\in H^{-1}(\Om;w)$,
$\Sigma_0=\Ga_0\ts (0,T)$, $\Sigma=\pt\Om\ts (0,T)$,   $\vp^0\in H_0^1(\Om;w)$ and $\vp^1\in L^2(\Om)$.

	\begin{theorem}\label{12.04.T4}
		The system \eqref{12.16.2*} is exactly null controllable for $T>\f{2b}{a}$.
	\end{theorem}

The paper is organized as follows. In Section \ref{S2}, we establish the existence of solutions to equation \eqref{12.14.1}. In Section \ref{S3}, we construct a sequence of uniformly hyperbolic equations that approximate the degenerate hyperbolic equation. In Section \ref{S4}, we first establish the observability inequality for the approximate uniformly hyperbolic equations and then deduce the corresponding observability inequality for the degenerate hyperbolic equation.

	\section{Existence of the solution to equation \eqref{12.14.1}}\label{S2}
	
	 In this section, we first introduce the weighted Sobolev spaces that are intricately linked to equation \eqref{12.14.1}. Subsequently, we employ spectral theory to demonstrate the existence of solutions for equation \eqref{12.14.1}. Finally, we provide a lemma that delineates an equivalent expression for the existence of solutions to equation \eqref{12.14.1}.
	
	\subsection{Solution spaces}

	 We introduce the definition
	\begin{equation*}
		H^{1}(\Omega; w) = \left\{u \in L^2(\Omega) \colon \int_\Om (\nabla u\cdot \nabla u)w\df x<+\iy\right\},
	\end{equation*}
	where $\nabla u=(\pt_{x_1}u,\cdots, \pt_{x_N}u)$.  The inner product and the norm on $H^1(\Omega; w)$
 are defined, respectively, as
	\begin{equation*}
		(u, v)_{H^1(\Omega; w)} =\int_\Om uv\df x+\int_\Om (\nabla u\cdot \nabla v)w\df x, \mbox{ and }  \|u\|_{H^1(\Om;w)}=(u,u)_{H^1(\Om;w)}^\f{1}{2}.
	\end{equation*}
	We define
	\begin{equation*}
		H_{0}^1(\Omega; w) = \text{the closure of } \mathcal{D}(\Omega) \text{ in } H^1(\Omega; w),
	\end{equation*}
	where $\mathcal{D}(\Omega) = C_0^\infty(\Omega)$ is the space of test functions.
	We denote by $H^{-1}(\Omega; w)$ the dual space of $H_{0}^1(\Omega; w)$ with pivot space $L^2(\Om)$, which is a subspace of $\mathcal{D}'(\Omega)$.
	
	Now, we define
	\begin{equation*}
		H^2(\Om;w)=\left\{u\in H^1(\Om;w)\colon \int_\Om (\mcA u)^2\df x<+\iy\right\},
	\end{equation*}
where $\mcA$ is given by \dref{gbz1}.
	The inner product and the norm on $H^2(\Om;w)$  are given, respectively, by
	\begin{equation*}
		(u,v)_{H^2(\Om;w)}=(u,v)_{H^1(\Om;w)}+\int_\Om (\mcA u)(\mcA v)\df x, \mbox{ and } \|u\|_{H^2(\Om;w)}=(u,u)_{H^2(\Om;w)}^\f{1}{2}.
	\end{equation*}

	It is  well-known the spaces
\begin{equation*}
		(H_0^1(\Om;w),(\cdot,\cdot)_{H_0^1(\Om;w)}),\ (H^1(\Om;w),(\cdot,\cdot)_{H^1(\Om;w)}),\mbox{ and } (H^2(\Om;w), (\cdot, \cdot)_{H^2(\Om;w)})
	\end{equation*}
	are Hilbert spaces (\cite{GC,Heinonen}).

	The following lemma is a version of Hardy's inequality.
	
	\begin{lemma}\label{08.15.L1}
		Let $N \geq 2$ and $\alpha \in (0, 2)$. Then, for all $u \in H_0^1(\Omega; w)$, the following inequality is satisfied:
		\begin{equation*}
			(N - 2 + \alpha)^2\int_\Om |x|^{\al-2}u^2\df x \leq 4\int_\Om |x|^\al (\nabla u\cdot \nabla u)\df x.
		\end{equation*}
	\end{lemma}
	
	\begin{proof}
		  The first part of this lemma can be deduced from \cite[Lemma 3.1]{Wu} or \cite[Proposition 2.1 (1)]{Stuart}. A similar proof can be found in the proof of Lemma \ref{08.16.L2} below. 
	\end{proof}
	\begin{remark}\label{08.16.R1}
		From Lemma \ref{08.15.L1} and given that $\alpha \in (0, 2)$, we can derive
		\begin{equation}\label{12.09.2}
			\int_\Om u^2\df x\leq M^{2-\al}\int_\Om |x|^{\al-2}u^2\df x\leq \f{4M^{2-\al}}{(N-2+\al)^2}\int_\Om w|\nabla u|^2\df x, \text{ with } M := \sup_{x \in \Omega} |x| + 1,
		\end{equation}
		 This is  the Poincar\'{e} inequality.
		In particular, the norm
		\begin{equation}\label{08.19.2}
			\|u\|_{H_0^1(\Omega; w)} = \left(\int_\Omega (\nabla u \cdot \nabla u) w \, \mathrm{d}x\right)^{\frac{1}{2}}
		\end{equation}
		is an equivalent norm in $H_0^1(\Omega; w)$.  From now on, we use \eqref{08.19.2} to define the norm of $H_0^1(\Omega; w)$.
	\end{remark}
	
The following Lemma \ref{08.15.L4} is just the  \cite[Theorem 3.4]{Wu} or \cite[Proposition 2.1 (5)]{Stuart}.

	\begin{lemma}\label{08.15.L4}
		The embedding $H_0^1(\Omega; w) \hookrightarrow L^2(\Omega)$ is compact.
	\end{lemma}
	
 The following lemma demonstrates that
  $C_0^\infty(\Omega)$ is dense in $D(\mathcal A)$ because  $C_0^\iy(\Om)$ is dense in $H_0^1(\Om;w)$.
	
\begin{lemma}\label{12.01.L1}
	 $C_0^\iy(\Om)\s D(\mcA)$ where $\mcA$ is defined as \dref{gbz1}.
\end{lemma}

\begin{proof}
	For each $u\in C_0^\iy(\Om)$, it has  $u\in H_0^1(\Om;w)$ and
	\begin{equation*}
		\begin{split}
			\int_\Om (\mcA u)^2\df x
			&=\int_\Om \left[\Div(|x|^\al \nabla u)\right]^2\df x=\int_\Om \left(|x|^\al \De u+\al |x|^{\al-2}x\cdot \nabla u\right)^2\df x\\
			&\leq C\sup_{x\in\Om}|\De u(x)|^2+C\sup_{x\in\Om}|\nabla u(x)|^2\int_0^M  r^{2\al+N-3}\df r<+\iy,
		\end{split}
	\end{equation*}
	where $M$ is a constant that is defined in \eqref{12.09.2}, and  the positive constant $C$ depends only on $\al, M$, $|\Om|$ and $u$. 
\end{proof}

\subsection{Spectrum}

Now, we consider the following degenerate elliptic equation
\begin{equation}\label{02.12.1}
	\begin{cases}
		\mcA u=g, &\mbox{in }\Om,\\
		u=0, &\mbox{on }\pt\Om, 
	\end{cases}
\end{equation}
where $g\in L^2(\Om)$. 

\begin{definition}\label{02.12.D1}
	We call $u\in H_0^1(\Om;w)$ is a weak solution of \eqref{02.12.1} with respect to $g$, provided
	\begin{equation*}
		\int_\Om w\nabla u\cdot \nabla v\df x=\int_\Om gv\df x
	\end{equation*}
	for all $v\in H_0^1(\Om;w)$. 
\end{definition}

\begin{lemma}\label{02.12.L1}
	Let $\al\in (0,2)$. The equation \eqref{02.12.1} has a unique solution $u\in D(\mcA)$ with respect to $g$. Moreover, there exists a positive constant $C$, depending only on $\al$, such that 
	\begin{equation*}
		\int_\Om w\nabla u\cdot \nabla u\df x\leq C\int_\Om g^2\df x. 
	\end{equation*}
\end{lemma}

\begin{proof}
	Denote 
	\begin{equation*}
		B[u,v]=\int_\Om w\nabla u\cdot \nabla v\df x, \mbox{ for all } u, v\in H_0^1(\Om;w). 
	\end{equation*}
	Then $B[u,u]=\int_\Om w\nabla u\cdot \nabla u\df x$ and 
	\begin{equation*}
		|B[u,v]|\leq \|u\|_{H_0^1(\Om;w)}\|v\|_{H_0^1(\Om;w)}.
	\end{equation*}
	Note that $f\in L^2(\Om)\s H^{-1}(\Om;w)$ by 
	\begin{equation*}
		\int_\Om fv\df x\leq \|f\|_{L^2(\Om)}\|v\|_{L^2(\Om)}\leq C\|f\|_{L^2(\Om)}\|v\|_{H_0^1(\Om;w)}
	\end{equation*}
	according to Remark \ref{08.16.R1}. From the Lax-Milgram theorem, we get there exists a unique $u\in H_0^1(\Om;w)$ such that $B[u,v]=(f,v)_{L^2(\Om)}$. 
	
	Finally, from $\mcA u=g$ we get $u\in H^2(\Om;w)$. This complete the proof of this lemma. 
\end{proof}

\begin{notation}\label{02.04.N1}
 Based on Remark \ref{08.16.R1} and Lemmas \ref{08.15.L4} and \ref{02.12.L1}, we can conclude that the degenerate partial differential operator
$\mcA$, as defined by \dref{gbz1}, possesses a discrete point spectrum given by
\begin{equation}\label{11.14.3}
	0<\la_1\leq \la_2\leq \la_3\leq  \cdots \ra +\iy.
\end{equation}
 This means that for each  $\la_n\ (n\in\N\equiv\{1,2,\cdots\})$, the following equation holds:
\begin{equation}\label{02.12.2}
	\begin{cases}
		\mcA \Phi_n=\la_n \Phi_n, & \mbox{in }\Om, \\
		\Phi_n=0, &\mbox{in }\pt\Om.
	\end{cases}
\end{equation}
 Furthermore, from \eqref{12.09.2}, we derive
\begin{equation}\label{11.30.9}
	\la_1=\inf_{0\neq u\in H_0^1(\Om;w)}\f{\int_\Om w\nabla u\cdot \nabla u\df x}{\int_\Om u^2\df x}\geq \f{(N-2+\al)^2}{4M^{2-\al}}.
\end{equation}
  For subsequent discussions, we let $\Phi_n(x)\in D(\mcA)$ denote the
$n$-th eigenfunction  of $\mcA$  corresponding to the eigenvalue $\la_n$ where $n\in\N$.
 The set $\{\Phi_n\}_{n\in\N}$   forms an orthonormal basis for  $L^2(\Om)$, and, in addition, constitutes an orthogonal subset of $H_0^1(\Om;w)$.
  For further details, refer to \cite[Theorem 7, Appendix D, p.728]{Evans} or the following  Lemma \ref{06.28.L1}.
\end{notation}

\begin{lemma}\label{06.28.L1}
	Let $u=\sum_{i=1}^\iy u_i \Phi_i\in H_0^1(\Om;w)$ with $u_i=(u,\Phi_i)_{L^2(\Om)}$ for all $i\in\N$.
 Then, $\nabla u=\sum_{i=1}^\iy u_i\nabla \Phi_i$ and $\|u\|_{H_0^1(\Om;w)}=(\sum_{i=1}^\iy u_i^2\la_i)^\f{1}{2}$,  Moreover,
	\begin{equation*}
		u\in H^2(\Om;w)\Lra \sum_{i=1}^\iy u_i^2\la_i^2<\iy,
	\end{equation*}
	and
	\begin{equation*}
		\mcA u=\sum_{i=1}^\iy u_i\la_i\Phi_i \mbox{ and } \|\mcA u\|_{L^2(\Om)}=\left(\sum_{i=1}^\iy u_i^2\la_i^2\right)^\f{1}{2}.
	\end{equation*}
\end{lemma}

\begin{proof}
	 From \eqref{11.14.3}, we obtain
	\begin{equation}\label{06.04.3}
		(\Phi_k,\Phi_l)_{H_0^1(\Om;w)}=\int_\Om (\nabla \Phi_k\cdot\nabla \Phi_l)w\df x=\de_{kl}\la_k \mbox{ for } k,l\in \N,
	\end{equation}
	where $\de_{kl}$  denotes the Kronecker delta function, which is defined as $\de_{kl}=1$ when  $k=l$ and $\de_{kl}=0$ when  $k\neq l$.
	 We now proceed to show that $\{\la_k^{-\f{1}{2}}\Phi_k\}_{k=1}^\iy$ forms an orthonormal basis for $H_0^1(\Om;w)$.
	 From \eqref{06.04.3}, it is readily apparent that  $\{\la_k^{-\f{1}{2}}\Phi_k\}_{k=1}^\iy$ is an orthonormal subset of $H_0^1(\Om;w)$. To establish that it is indeed a basis, we employ a proof by contradiction. Assume that there exists a nonzero function  $0\neq u\in H_0^1(\Om;w)$ such that
	\begin{equation*}
		\int_\Om (\nabla \Phi_k\cdot \nabla u)w\df x=0 \mbox{ for all }k\in\N.
	\end{equation*}
	Since  $\{\Phi_k\}_{k\in\N}$ forms  an orthonormal basis for $L^2(\Om)$, for any $u\in H_0^1(\Om;w)$, we can express $u$ as
	\begin{equation}\label{06.04.4}
		u=\sum_{k=1}^\iy d_k\Phi_k, \mbox{ where }  d_k=(u, \Phi_k)_{L^2(\Om)}, k\in\N.
	\end{equation}
	Then, from \eqref{02.12.2},   we have
	\begin{equation*}
		0=(u,\Phi_k)_{H_0^1(\Om;w)}=\int_\Om (\nabla \Phi_k\cdot \nabla u)w\df x=\la_k\int_\Om \Phi_ku\df x =\la_kd_k,
	\end{equation*}
	which implies that $d_k=0$.  This leads to the conclusion that  $u=0$, which is a contradiction.
	
	From the above discussion, as $u\in H_0^1(\Om;w)$,  we have
	\begin{equation*} u=\sum_{i=1}^\iy e_i\left(\la_i^{-\f{1}{2}} \Phi_k\right), \mbox{ and } \|u\|_{H_0^1(\Om;w)}^2=\sum_{i=1}^\iy e_i^2,
	\end{equation*}
	where
	\begin{equation*}
		e_i=\int_\Om \left(\nabla u\cdot \nabla\left[ \la_i^{-\f{1}{2}}\Phi_i\right]\right) w\df x=\la_i^{-\f{1}{2}}\int_\Om (\nabla u\cdot \nabla \Phi_i)w\df x=\la_i^\f{1}{2}\int_\Om u\Phi_i\df x=\la_i^\f{1}{2}u_i.
	\end{equation*}
	Hence,  $\nabla u=\sum_{i=1}^\iy u_i\nabla\Phi_i$. Moreover, $\|u\|_{H_0^1(\Om;w)}=(\sum_{i=1}^\iy \la_iu_i^2)^\f{1}{2}$.
	
	Let $u\in H^2(\Om;w)$. Define $\vp_n=\sum_{i=1}^n u_i\Phi_i\in H_0^1(\Om;w)\cap H^2(\Om;w)$ for each $n\in\N$. Then,  from
	\begin{equation*}
		\begin{split}
			(\mcA u, \mcA\vp_n)_{L^2(\Om)}
			&=\sum_{i=1}^n u_i\la_i \int_\Om (\mcA u)\Phi_i\df x=\sum_{i=1}^n u_i\la_i\int_\Om w\nabla u\cdot \nabla \Phi_i\df x\\
			&=\sum_{i=1}^n u_i\la_i^2\int_\Om u\Phi_i\df x=\sum_{i=1}^n u_i^2\la_i^2
		\end{split}
	\end{equation*}
	and $\|\mcA\vp_n\|_{L^2(\Om)}^2=\sum_{i=1}^n u_i^2\la_i^2$, by the Cauchy inequality, we obtain
	\begin{equation*}
		\sum_{i=1}^nu_i^2\la_i^2\leq \|\mcA u\|_{L^2(\Om)}^2.
	\end{equation*}
	for all $n\in\N$. This implies that $\sum_{i=1}^\iy u_i^2\la_i^2\leq \|\mcA u\|_{L^2(\Om)}^2<\iy$.

	Let $\sum_{i=1}^\iy u_i^2\la_i^2<\iy$. For each  $\vp\in C_0^\iy(\Om)$,  we have
	\begin{equation*}
		\begin{split}
			(\mcA u, \vp)_{L^2(\Om)}
			&=\int_\Om (\nabla u\cdot \nabla\vp)w\df x=\sum_{i=1}^\iy u_i \int_\Om (\nabla\Phi_i\cdot \nabla\vp)w\df x=\sum_{i=1}^\iy u_i\la_i\int_\Om \Phi_i\vp\df x.
		\end{split}
	\end{equation*}
	Note that
	\begin{equation*}
		\int_\Om \left(\sum_{i=1}^n u_i\la_i\Phi_i\right)^2\df x=\sum_{i=1}^n u_i^2\la_i^2\leq \sum_{i=1}^\iy u_i^2\la_i^2<\iy \mbox{ for all } n\in\N,
	\end{equation*}
	that is,  $\sum_{i=1}^\iy u_i\la_i\Phi_i\in L^2(\Om)$.  Hence,
	\begin{equation*}
		(\mcA u, \vp)_{L^2(\Om)}=\left(\sum_{i=1}^\iy u_i\la_i\Phi_i, \vp\right)_{L^2(\Om)}.
	\end{equation*}
	This implies that $\mcA u=\sum_{i=1}^\iy u_i\la_i\Phi_i$ and  $\|\mcA u\|_{L^2(\Om)}\leq (\sum_{i=1}^\iy u_i^2\la_i^2)^\f{1}{2}$. We complete the proof of this lemma. 
\end{proof}



\subsection{Existence of the weak  solution}

 This section focuses on establishing the existence of weak solutions for equation \eqref{12.14.1}.

	\begin{definition}\label{11.13.D1}
		 A function
		\begin{equation*}
			\vp\in L^2(0,T; H_0^1(\Om;w)), \mbox{ with } \pt_t\vp\in L^2(Q), \ \pt_{tt}\vp\in L^2(0,T; H^{-1}(\Om;w))
		\end{equation*}
		is said to be a weak solution of \eqref{12.14.1} with respect to $(\vp^0,\vp^1,f)$   if it satisfies the following conditions:
		
		(i) For every  $\psi\in H_0^1(\Om;w)$  and almost every  $t\in [0,T]$, the following equation holds:
		\begin{equation}\label{11.18.4}
			\lg \pt_{tt}\vp, \psi\rg_{H^{-1}(\Om;w), H_0^1(\Om;w)}+\int_\Om w\nabla \vp \cdot \nabla \psi \df x=\int_\Om f\psi\df x;
		\end{equation}
		
		(ii) The initial conditions are satisfied as $\vp(0)=\vp^0, \pt_t\vp(0)=\vp^1$.
	\end{definition}

The following lemma is just  \cite[Lemma 2.3, p. 61)]{Bellassoued}.

\begin{lemma}\label{11.30.L2}
	Let $(\mcM,g)$ be a $C^m$-Riemannian manifold with compact boundary $\pt \mcM$. Then,
 there exists a $C^{m-1}$-vector field $\bs{n}$ such that
	\begin{equation*}
		\bs{n}(x)=\nu(x), \ x\in\pt \mcM, \mbox{ and } |\bs{n}(x)|\leq 1, \ x\in \mcM,
	\end{equation*}
	where $\nu$  denotes the unit outward normal vector to $\pt \mcM$.
\end{lemma}

We are now in a position to  demonstrate the existence of a weak solution to equation \eqref{12.14.1}.
	
\begin{theorem}\label{11.14.T1}
	Given  $\vp^0\in H_0^1(\Om;w), \vp^1\in L^2(\Om)$ and $f\in L^2(Q)$, there exists a unique weak solution
$\vp$ characterized by
	\begin{equation}\label{11.14.4}
		\vp\in L^2(0,T; H_0^1(\Om;w))\cap H^1(0,T; L^2(\Om))\cap H^2(0,T; H^{-1}(\Om;w))
	\end{equation}
	for the equation \eqref{12.14.1}.   Moreover,
	\begin{equation}\label{12.01.1}
		\f{\pt \vp}{\pt \nu}\in L^2(\pt Q).
	\end{equation}
	 Additionally, the following estimate holds:
	\begin{equation}\label{11.18.1}
		\begin{split}
		&\esssup_{t\in [0,T]}\left(\|\vp(t)\|_{H_0^1(\Om;w)}+\|\pt_t\vp(t)\|_{L^2(\Om)}\right)+\|\pt_{tt}\vp\|_{L^2(0,T; H^{-1}(\Om;w))}+\left\|\f{\pt \vp}{\pt \nu}\right\|_{L^2(\pt Q)}\\
		&\leq C\left(\|\vp^0\|_{H_0^1(\Om;w)}+\|\vp^1\|_{L^2(\Om)}+\|f\|_{L^2(Q)}\right),
		\end{split}
	\end{equation}
	where the constant $C>0$ is only dependent on $\al, T$ and $M$.
	
	Furthermore, if $f\in H^1(0,T; L^2(\Om)), \vp^0\in D(\mcA), \vp^1\in H_0^1(\Om;w)$, then
	\begin{equation}\label{11.14.6}
		\vp\in L^2(0,T; D(\mcA))\cap H^1(0,T; H_0^1(\Om;w))\cap H^2(0,T; L^2(\Om)).
	\end{equation}
	 Additionally, the following estimate is valid:
	\begin{equation}\label{11.18.2}
		\begin{split}
			&\esssup_{t\in [0,T]}\left(\|\vp(t)\|_{D(\mcA)}+\|\vp(t)\|_{H^2(\mcO(\pt \Om;2R_0))}+\|\pt_t\vp(t)\|_{H_0^1(\Om;w)}\right)+\left\|\f{\pt \vp}{\pt \nu}\right\|_{H^1(0,T; \pt\Om)}\\
			&\leq C\left(\|\vp^0\|_{D(\mcA)}+\|\vp^1\|_{H_0^1(\Om;w)}+\|f\|_{H^1(0,T; L^2(\Om))}\right),
		\end{split}
	\end{equation}
	where $\mcO(\pt\Om;2R_0)=\{x\in\Om\colon \dist(x,\pt \Om)<2R_0\}$,  and the constant $C>0$ depends only on $\al, R_0, T$ and $\mcO(\pt\Om;3R_0)$.
\end{theorem}

\begin{proof}
	 The proof of this theorem proceeds through the following nine steps.

	{\it Step 1: Galerkin Method.}  Let $m\in\N$. 
	\begin{equation}\label{11.30.4}
		y^m(x, t)=\sum_{n=1}^m d_n^m(t)\Phi_n(x).
	\end{equation}
	Consider the following ordinary differential equations
	\begin{equation}\label{11.30.5}
		(d_n^m)''+\la_nd_n^m=f_n^m, \ n=1,\cdots, m
	\end{equation}
	with initial conditions
  \begin{equation}\label{11.30.3}
		d_n^m(0)=(\vp^0, \Phi_n)_{L^2(\Om)}, \quad (d_n^m)'(0)=(\vp^1, \Phi_n)_{L^2(\Om)}, \ n=1,\cdots, m,
	\end{equation}
	where
	\begin{equation}\label{02.04.1}
		f^m=\sum_{n=1}^m f_n^m\Phi_n, \quad f_n^m=(f,\Phi_n)_{L^2(\Om)} \mbox{ for } n=1,\cdots, m.
	\end{equation}
	By the standard theory for ordinary differential equations, there exists a unique function $d^m(t)=(d_1^m(t),\cdots, d_m^m(t))$, satisfying \eqref{11.30.5} and \eqref{11.30.3} for $0\leq t\leq T$. It is clear that \eqref{11.30.5} is equivalent to
	\begin{equation}\label{02.04.2}
		\left(\pt_{tt} y^m, \Phi_n\right)_{L^2(\Om)}+(y^m, \Phi_n)_{H_0^1(\Om;w)}=(f^m,\Phi_n)_{L^2(\Om)}, \ n=1,\cdots, m
	\end{equation}
	by \eqref{06.04.3} and \eqref{11.30.4}.
	
{\it Step 2: Energy Estimate.}

 Multiply equation \eqref{11.30.5} by  $(d_n^m)'(t)$ and sum over $n=1,\cdots, m$, we arrive at
\begin{equation}\label{11.30.6}
	\begin{split}
		\int_\Omega (\partial_{tt}y^m)(\partial_ty^m)\,dx+\int_\Omega w\nabla y^m\cdot \nabla \partial_ty^m\,dx=\int_\Omega f^m \partial_ty^m\,dx.
	\end{split}
\end{equation}
 It is worth noting that 
\begin{equation*}
	\int_\Omega (\partial_{tt} y^m)(\partial_ty^m)\,dx=\frac{1}{2}\frac{d}{dt}\|\partial_ty^m\|_{L^2(\Omega)}^2, \quad \int_\Omega w\nabla y^m\cdot \nabla \partial_t y^m\,dx=\frac{1}{2}\frac{d}{dt}\|y^m\|_{H_0^1(\Omega;w)}^2. 
\end{equation*}
 Combining these results with \eqref{11.30.6}, we obtain 
\begin{equation*}
	\frac{d}{dt}\left(\|\partial_ty^m\|_{L^2(\Omega)}^2+\|y^m\|_{H_0^1(\Omega;w)}^2\right)\leq \|f^m\|_{L^2(\Omega)}^2+\|\partial_ty^m\|_{L^2(\Omega)}^2.
\end{equation*}
Consequently, by employing Gronwall's inequality \cite[Appendix B.2 (j), p. 708]{Evans} along with \eqref{11.30.3}, \eqref{02.04.1}, and Lemma \ref{06.28.L1}, we have for almost every $t\in [0,T]$  
 that
\begin{equation}\label{11.30.7}
	\begin{split}
	\|\partial_ty^m\|_{L^2(\Omega)}^2+\|y^m\|_{H_0^1(\Omega;w)}^2
	&\leq e^t\left(\|\partial_ty^m(0)\|_{L^2(\Omega)}^2+\|y^m(0)\|_{H_0^1(\Omega;w)}^2+\|f^m\|_{L^2(Q)}^2\right)\\
	&\leq e^T\left(\|\varphi^1\|_{L^2(\Omega)}^2+\|\varphi^0\|_{H_0^1(\Omega;w)}^2+\|f\|_{L^2(Q)}^2\right).
	\end{split}
\end{equation}

  Furthermore, for any $v\in H_0^1(\Omega;w)$ satisfying  $\|v\|_{H_0^1(\Omega;w)}\leq 1$, 
   we can decompose $v$ as $v=v^1+v^2$, where $v^1\in \Span \{\Phi_k \}_{k=1}^m$ and $(v^2,\Phi_k)_{L^2(\Omega)}=0$ for $k=1,\cdots, m$. Based on \eqref{11.30.4} and \eqref{11.30.5}, we get
\begin{equation*}
	\begin{split}
		\langle \partial_{tt}y^m, v\rangle_{H^{-1}(\Omega;w),H_0^1(\Omega;w)}&=(\partial_{tt}y^m, v)_{L^2(\Omega)}=(\partial_{tt}y^m, v^1)_{L^2(\Omega)}\\&=(f^m, v^1)_{L^2(\Omega)}-(y^m,v^1)_{H_0^1(\Omega;w)}.
	\end{split}
\end{equation*}
Thus, by invoking Remark \ref{08.16.R1} and the fact that
$\|v^1\|_{H_0^1(\Omega;w)}\leq 1$,  we have for almost every $t\in [0,T]$ that 
\begin{equation*}
	\|\partial_{tt}y^m(t)\|_{H^{-1}(\Omega;w)}\leq C\|f^m(t)\|_{L^2(\Omega)}+\|y^m(t)\|_{H_0^1(\Omega;w)},
\end{equation*}
where the constant $C>0$ depends only on $\alpha$ and $M$.   As a result, according to \eqref{11.30.7}, we obtain 
\begin{equation}\label{11.30.8}
	\int_0^T\|\partial_{tt}y^m\|_{H^{-1}(\Omega;w)}^2\,dt\leq C\left(\|\varphi^0\|_{H_0^1(\Omega;w)}^2+\|\varphi^1\|_{L^2(\Omega)}^2+\|f\|_{L^2(Q)}^2\right),
\end{equation}
 where the constant $C>0$ depends only on $\alpha, T$, and $M$.

{\it Step 3: Approximation or weak convergence.}

 Based on \eqref{11.30.7} and \eqref{11.30.8}, we can extract a subsequence from
$\{y^m\}_{m\in\mathbb{N}}$,  which we shall continue to denote by the same symbol, along with a function
 $y\in L^2(0,T; H_0^1(\Omega;w))$  
 satisfying the following weak convergence properties:
\begin{equation}\label{02.04.3}
	\begin{split}
		y^m
		&\rightharpoonup  y  \; \text{ weak}^* \hbox{  in } L^\infty (0,T; H_0^1(\Omega;w)),\\
		\partial_ty^m
		&\rightharpoonup \partial_t y  \; \text{ weak}^* \hbox{  in } L^\infty (0,T; L^2(\Omega)),\\
		\partial_{tt} y^m
		&\rightharpoonup \partial_{tt}y  \quad \text{weakly in } L^2(0,T; H^{-1}(\Omega;w)).
		\end{split}
\end{equation}
 Furthermore, leveraging \eqref{11.30.7}, we derive the estimate 
\begin{equation}\label{12.01.2}
	\begin{split}
		&\esssup_{t\in [0,T]}\left(\|y(t)\|_{H_0^1(\Omega;w)}+\|\partial_ty(t)\|_{L^2(\Omega)}\right)+\|\partial_{tt}y\|_{L^2(0,T; H^{-1}(\Omega;w))}\\
		&\leq C\left(\|\varphi^0\|_{H_0^1(\Omega;w)}+\|\varphi^1\|_{L^2(\Omega)}+\|f\|_{L^2(Q)}\right),
	\end{split}
\end{equation}
where the constant $C>0$ depends only on $\alpha, T$, and $M$.   Additionally, we observe that 
\begin{equation}\label{02.04.4}
	y\in C([0,T]; L^2(\Omega)) \hbox{ and }  \partial_ty\in C([0,T]; H^{-1}(\Omega;w)).
\end{equation}

{\it Step 4: The function $y$ is a weak solution of \eqref{12.14.1}.}

 We proceed to demonstrate that
$y$
 fulfills Definition \ref{11.13.D1} (i).

Fix an arbitrary integer
$k\in\mathbb{N}$, and select a function  $\psi\in C^2([0,T]; H_0^1(\Omega;w))$
 that takes the form
 \begin{equation}\label{11.30.10}
	\psi(t)=\sum_{n=1}^k \psi_n(t)\Phi_n.
\end{equation}
 By choosing
$m\geq k$, multiplying \eqref{11.30.5} by $\psi_n(t)$ for $n=1,\cdots, k$,
 and then integrating with respect to  $t\in [0,T]$, we arrive at
\begin{equation}\label{02.04.5}
	\int_0^T \left(\langle \partial_{tt}y^m, \psi\rangle_{H^{-1}(\Omega;w), H_0^1(\Omega;w)}+(y^m, \psi)_{H_0^1(\Omega;w)} \right)\,dt=\int_0^T (f^m, \psi)_{L^2(\Omega)}\,dt.
\end{equation}
 Based on \eqref{02.04.3}, taking the limit as
$m\to\infty$, we obtain
\begin{equation}\label{11.30.12}
	\begin{split}
		\int_0^T\left(\langle \partial_{tt}y, \psi\rangle_{H^{-1}(\Omega;w), H_0^1(\Omega;w)}+(y, \psi)_{H_0^1(\Omega;w)}\right)\,dt=\int_0^T (f,\psi)_{L^2(\Omega)}\,dt.
	\end{split}
\end{equation}
 It is important to note that the set of functions
$\psi(t)$ 
 of the form \eqref{11.30.10} is dense in $L^2(0,T; H_0^1(\Omega;w))$. Consequently, \eqref{11.30.12} holds for all
$\psi\in L^2(0,T; H_0^1(\Omega;w))$. Combining this with \eqref{11.30.12}, we derive
\begin{equation*}
	\langle \partial_{tt}y, v\rangle_{H^{-1}(\Omega;w),H_0^1(\Omega;w)}+(y,v)_{H_0^1(\Omega;w)}=(f,v)_{L^2(\Omega)}
\end{equation*}
for all $v\in H_0^1(\Omega;w)$ and almost every $t\in [0,T]$.  This establishes the validity of Definition \ref{11.13.D1} (ii).

Next, we proceed to verify that $y(0)=\varphi^0$ and $\partial_ty(0)=\varphi^1$.

For any $h\in C^2([0,T]; H_0^1(\Omega;w))$ satisfying  $h(T)=\partial_th(T)=0$,   from \eqref{11.30.12} and \eqref{02.04.4}, we have
\begin{equation}\label{01.13.15}
	\begin{split}
		&\iint_Qy \partial_{tt}h\,dx\,dt+\iint_Q w\nabla y\cdot \nabla h\,dx\,dt\\
		&=\iint_Q fh\,dx\,dt-(y(0),\partial_th(0))_{L^2(\Omega)}+\langle \partial_ty(0),h(0)\rangle_{H^{-1}(\Omega;w),H_0^1(\Omega;w)}.
	\end{split}
\end{equation}
Similarly, from \eqref{02.04.5} we obtain 
\begin{equation*}
	\begin{split}
		&\iint_Qy^k \partial_{tt}h\,dx\,dt+\iint_Q w\nabla y^k\cdot \nabla h\,dx\,dt\\
		&=\iint_Q f^kh\,dx\,dt-(y^k(0),\partial_th(0))_{L^2(\Omega)}+(\partial_ty^k(0), h(0))_{L^2(\Omega)},
	\end{split}
\end{equation*}
 which, when combined with \eqref{11.30.3}, \eqref{02.04.1}, and \eqref{02.04.3}, leads to
\begin{equation*}
	\begin{split}
		&\iint_Qy \partial_{tt}h\,dx\,dt+\iint_Q w\nabla y\cdot \nabla h\,dx\,dt\\
		&=\iint_Q fh\,dx\,dt-(\varphi^0,\partial_th(0))_{L^2(\Omega)}+(\varphi^1,h(0))_{L^2(\Omega)}.
	\end{split}
\end{equation*}
This, together with \eqref{01.13.15}, implies
\begin{equation*}
	y(0)=\varphi^0, \quad \text{and} \quad \partial_ty(0)=\varphi^1.
\end{equation*}
 This completes the proof of Definition \ref{11.13.D1} (ii).

{\it Step 5: Uniqueness.}

 Assume that $\varphi^0=\varphi^1=f=0$.  Let
$y$
 denote the solution of \eqref{12.14.1} corresponding to the parameter values
$(\varphi^0,\varphi^1,f)=(0,0,0)$. For a fixed $s\in [0,T]$, define the function
\begin{equation*}
	h(t)=
	\begin{cases}
		\int_t^sy(\tau)\,d\tau, & t\in [0,s],\\
		0, &t\in [s,T].
	\end{cases}
\end{equation*}
It follows that  $h(t)\in H_0^1(\Omega;w)$ for almost every $t\in [0,T]$,  and the following equality holds:
\begin{equation*}
	\int_0^s\langle\partial_{tt}y, h\rangle_{H^{-1}(\Omega;w),H_0^1(\Omega;w)}\,dt+\int_0^sw\nabla y \cdot \nabla h\,dt=0.
\end{equation*}
 Observe that  $\partial_ty(0)=h(s)=0$.  
 By integrating by parts with respect to
$t$
 and utilizing the fact that $\partial_th=-y$ for $0\leq t\leq s$, we obtain
\begin{equation*}
	\int_0^s(\partial_ty,y)_{L^2(\Omega)}\,dt-\int_0^s(\partial_th,h)_{H_0^1(\Omega;w)}\,dt=0.
\end{equation*}
 This can be rewritten as 
\begin{equation*}
	\frac{1}{2}\int_0^s\frac{d}{dt}\|y(t)\|_{L^2(\Omega)}^2\,dt-\frac{1}{2}\int_0^s\frac{d}{dt}\|h\|_{H_0^1(\Omega;w)}^2\,dt=0.
\end{equation*}
Given that  $y(0)=0$ and $h(s)=0$, we deduce that 
\begin{equation*}
	\|y(s)\|_{L^2(\Omega)}^2+\|h(0)\|_{H_0^1(\Omega;w)}^2=\|y(0)\|_{L^2(\Omega)}^2+\|h(s)\|_{H_0^1(\Omega;w)}^2=0.
\end{equation*}
 Consequently, $y(s)=0$. Since $s\in [0,T]$ is arbitrary, it follows that $y=0$. This establishes the uniqueness of the solution.

{\it Step 6:  Improved regularity.}

  We further assume that  $\pt_tf\in L^2(Q)$, that is, $f\in H^1(0,T; L^2(\Om))$. Let $\ol y^m=\pt_ty^m$. 
  From \eqref{11.30.4}, \eqref{11.30.5} and \eqref{02.04.2}, we obtain 
\begin{equation*}
	(\pt_{tt}\ol y^m, \Phi_k)_{L^2(\Om)}+(\ol y^m, \Phi_k)_{H_0^1(\Om;w)}=(\pt_t f^m, \Phi_k)_{L^2(\Om)} \mbox{ for } 0\leq t\leq T \mbox{ and } k=1,\cdots, m.
\end{equation*}
By multiplying with  $(d_k^m)''(t)$,  and summing over  $k=1,\cdots, m$, we  have
\begin{equation*}
	(\pt_{tt}\ol y^m, \pt_t \ol y^m)_{L^2(\Om)}+(\ol y^m, \pt_t \ol y^m)_{H_0^1(\Om;w)}=(\pt_t f^m, \pt_t\ol y^m)_{L^2(\Om)}.
\end{equation*}
  Following the same reasoning as in \eqref{11.30.7}, we derive 
\begin{equation}\label{09.07.13}
	\begin{split}
		&\esssup_{0\leq t\leq T}\left(\|\pt_t\ol y^m\|_{L^2(\Om)}+\|\ol y^m\|_{H_0^1(\Om;w)}\right)\\
		&\leq  \|\ol y^m(0)\|_{H_0^1(\Om;w)}+\|\pt_t\ol y^m(0)\|_{L^2(\Om)}+\|\pt_tf^m\|_{L^2(Q)}.
	\end{split}
\end{equation}
Note that
\begin{equation*}
	\pt_t\ol y^m(0)=(\pt_{tt}y^m)(0)=\sum_{k=1}^m (d_k^m)''(0)\Phi_k=\sum_{k=1}^m f_k^m(0)\Phi_k-\sum_{k=1}^m\la_kd_k^m(0)\Phi_k;
\end{equation*}
and
\begin{equation*}
	\int_\Om \left(\sum_{k=1}^m f_k^m(0)\Phi_k\right)^2\df x\leq \|f(0)\|_{L^2(\Om)}^2\leq C\left(\|f\|_{L^2(Q)}^2+\|\pt_tf\|_{L^2(Q)}^2\right) 
\end{equation*}
by   \cite[Theorem 2 (iii) in Chapter 5.9.2, p. 302]{Evans} and $f\in H^1(0,T; L^2(\Om))$;   and
\begin{equation*}
	\int_\Om \left(\sum_{k=1}^m\la_kd_k^m(0)\Phi_k\right)^2\df x=\sum_{k=1}^m \la_k^2(\vp^0,\Phi_k)_{L^2(\Om)}^2\leq \|\mcA\vp^0\|_{L^2(\Om)}^2
\end{equation*}
by \eqref{11.30.3} and Lemma \ref{06.28.L1}.  We obtain
\begin{equation}\label{02.04.6}
	\|\pt_t\ol y^m(0)\|_{L^2(\Om)}\leq C\left(\|\vp^0\|_{D(\mcA)}+\|f\|_{H^1(0,T; L^2(\Om))}\right),
\end{equation}
where the constants $C>0$ depending only on $T$. Now, we have
\begin{equation}\label{02.04.7}
	\begin{split}
		\|\ol y^m(0)\|_{H_0^1(\Om;w)}^2
		&=\|\pt_ty^m(0)\|_{H_0^1(\Om;w)}^2 \\
		&=\left(\sum_{n=1}^m (\vp^1,\Phi_n)_{L^2(\Om)}\Phi_n, \sum_{n=1}^m(\vp^1,\Phi_n)_{L^2(\Om)}\Phi_n\right)_{H_0^1(\Om;w)}\\
		&\leq \sum_{n=1}^m (\vp^1,\Phi_n)_{L^2(\Om)}^2\la_n\leq \|\vp^1\|_{H_0^1(\Om;w)}^2,
	\end{split}
\end{equation}
by \eqref{11.30.4}, \eqref{11.30.3}, and Lemma \ref{06.28.L1}. Combining \eqref{09.07.13}, \eqref{02.04.6}, and \eqref{02.04.7}, we obtain 
\begin{equation}\label{09.04.9}
	\begin{split}
		&\esssup_{0\leq t\leq T}\left(\|\pt_t\ol y^m\|_{L^2(\Om)}+\|\ol y^m\|_{H_0^1(\Om;w)}\right)\\
		&\leq C\left(\|\mcA \vp^0\|_{L^2(\Om)}+\|\vp^1\|_{H_0^1(\Om;w)}+\|f\|_{H^1(0,T; L^2(\Om))}\right)
	\end{split}
\end{equation}
 for almost every  $t\in [0,T]$, where the constants $C>0$ depends only on $T$.

Now, multiplying \eqref{02.04.2} by $\la_kd_k^m$, summing over $k=1,\cdots, m$, since $\mcA y^m=\sum_{n=1}^m y_n^m\la_n\Phi_n\in H_0^1(\Om;w)$, we get
\begin{equation}\label{09.07.14}
	\|\mcA y^m\|_{L^2(\Om)}^2=(y^m, \mcA y^m)_{H_0^1(\Om;w)}=(f^m-\pt_{tt}y^m, \mcA y^m)_{L^2(\Om)}.
\end{equation}
Choose  $\zeta\in C^\iy(\ol\R^N), 0\leq \zeta\leq 1$ such that
\begin{equation*}
	\begin{split}
		\zeta=1 \mbox{ on }\mcO(\pt\Om;2R_0),\mbox{ and } \zeta=0 \mbox{ on } \Om-\mcO(\pt\Om;3R_0),
	\end{split}
\end{equation*}
and
\begin{equation*}
	 |\nabla\zeta|\leq CR_0^{-1},  \mbox{ and } \left|\f{\pt^2\zeta}{\pt x_i\pt x_j}\right|\leq CR_0^{-2}, \mbox{ for all } i,j=1,\cdots, N,
\end{equation*}
where the constants $C>0$ are absolute. Then, $\xi=\zeta y^m$ is the solution of the following    equation
\begin{equation}\label{02.04.8}
	\mcA \xi=g \mbox{ in }\mcO(\pt\Om;3R_0),\quad \xi=0 \mbox{ in } \pt\mcO(\pt\Om;3R_0),
\end{equation}
where
\begin{equation*}
	g=-\zeta\pt_{tt}y^m+\zeta f^m-2w\nabla \zeta\cdot \nabla y^m-y^m\Div(w\nabla\zeta).
\end{equation*}
 It is evident that \eqref{02.04.8} is a uniformly elliptic equation with $g\in L^2(\Om)$. 
 By the classical elliptic regularity theory \cite[Theorem 1, Chapter 6.3.1, p.327 and Theorem 4, Chapter 6.3.2, p. 334]{Evans}, we have
\begin{equation*}
	\|\xi\|_{H^2(\mcO(\pt\Om;2R_0))}\leq C\left(\|g\|_{L^2(\Om)}+\|\xi\|_{L^2(\Om)}\right),
\end{equation*}
where   the constant $C>0$ depends only on $\al,R_0$ and $\mcO(\pt\Om;3R_0)$.
 Applying this to \eqref{09.07.14}, we obtain
\begin{equation}\label{02.13.3}
	\|y^m\|_{H^2(\mcO(\pt\Om;2R_0))}+\|\mcA y^m\|_{L^2(\Om)}\leq C\left(\|f^m\|_{L^2(\Om)}+\|\pt_{tt}y^m\|_{L^2(\Om)}+\|y^m\|_{H_0^1(\Om;w)}\right),
\end{equation}
where the constant $C>0$ depends only on $\al, R_0, T$ and $\mcO(\pt\Om;3R_0)$.

Recalling $\ol y^m=\pt_ty^m$,  from \eqref{11.30.7} and \eqref{09.04.9}, we obtain
\begin{equation*}
	\begin{split}
		&\esssup_{0\leq t\leq T}\left(\|y^m(t)\|_{H^2(\mcO(\pt\Om;2R_0))}+\|y^m(t)\|_{D(\mcA)}+\|\pt_ty^m(t)\|_{H_0^1(\Om;w)}+\|\pt_{tt}y^m(t)\|_{L^2(\Om)}\right)\\
		&\leq C\left(\|\vp^0\|_{D(\mcA)}+\|\vp^1\|_{H_0^1(\Om;w)}+\|f\|_{H^1(0,T; L^2(\Om))}\right).
	\end{split}
\end{equation*}
Letting $m\ra\iy$, we obtain 
\begin{equation}\label{09.07.16}
	\begin{split}
		&\esssup_{0\leq t\leq T}\left(\|\vp(t)\|_{H^2(\mcO(\pt\Om;2R_0))}+\|\vp(t)\|_{D(\mcA)}+\|\pt_t\vp(t)\|_{H_0^1(\Om;w)}+\|\pt_{tt}\vp(t)\|_{L^2(\Om)}\right)\\
		&\leq C\left(\|\vp^0\|_{D(\mcA)}+\|\vp^1\|_{H_0^1(\Om;w)}+\|f\|_{H^1(0,T; L^2(\Om))}\right),
	\end{split}
\end{equation}
where the constant $C>0$ depends only $\al, R_0, T$ and $\mcO(\pt\Om;3R_0)$.

 {\it Step 7: Hidden regularity.}

By invoking a density argument, it is sufficient to establish \eqref{12.01.1} for the cases where
$\vp^0\in D(\mcA),\vp^1\in H_0^1(\Om;w),f\in H^1(0,T; L^2(\Om))$.   According to Lemma \ref{11.30.L2}, there exists a smooth vector field
$\bs{n}$  defined on  $\Om$  satisfying
	\begin{equation*}
		\bs{n}(x)=\nu(x) \mbox{ for } x\in\pt\Om,  \mbox{ and } \supp \bs{n}\s \ol{\mcO(\pt\Om;2R_0)},
	\end{equation*}
	as well as 
	\begin{equation*}
		|\bs{n}(x)|\leq 1 \mbox{ for } x\in \Om,\mbox{ and } |\Div \bs{n}(x)|\leq CR_0^{-1} \mbox{ for } x\in \Om,
	\end{equation*}
where the positive constant $C$ depends solely on $\bs{n}$  (and thus, by Lemma \ref{11.30.L2}, only on 
$\mcO(\pt\Om;2R_0 $.

If the above conditions are not met, we can multiply
$\bs{n}$  
 by a cut-off function $\zeta\in C_0^\iy(\R^N)$ such that  $\zeta=1$ on $\ol{\mcO(\pt\Om;R_0)}$, 
$\zeta=0$ on $\Om\se \ol{\mcO(\pt\Om;2R_0)}$, and $|\nabla\zeta|\leq CR_0^{-1}$. Consequently,
$\zeta\bs{n}$  will be the desired function.

From {\it Step 6}, we know that  $\pt_{tt}\vp, \mcA\vp\in L^2(Q)$. By multiplying both sides of \eqref{12.14.1} by
$\bs{n}\cdot \nabla \vp$   and integrating over  $Q$ (the integrations are indeed on $\mcO(\pt\Om;2R_0)\ts (0,T)$ since $\bs{n}=0$ on $(\Om\se \ol{\mcO(\pt\Om;2R_0)})\ts (0,T)$), we obtain
	\begin{equation}\label{02.13.1}
		\begin{split}
			I
			&\equiv \iint_Q f(\bs{n}\cdot\nabla \vp)\df x\df t\\
			&=\iint_Q (\pt_{tt}\vp)(\bs{n}\cdot \nabla \vp)\df x\df t-\iint_Q \Div(w\nabla \vp) (\bs{n}\cdot \nabla \vp)\df x\df t\equiv I_1+I_2.
		\end{split}
	\end{equation}

Employing integration by parts and considering that
 $\supp \bs{n}\s \ol{\mcO(\pt\Om;2R_0)}$,  we have
	\begin{equation*}
		\begin{split}
			I_1
			&=\int_\Om (\pt_t\vp)(\bs n\cdot \nabla \vp)\df x\bigg|_{t=0}^{t=T}-\f{1}{2}\iint_Q \bs{n}\cdot \nabla |\pt_t\vp|^2\df x\df t\\
			&=\int_\Om (\pt_t\vp)(\bs n\cdot \nabla \vp)\df x\bigg|_{t=0}^{t=T}-\f{1}{2}\iint_{\pt Q} |\pt_t\vp|^2\bs{n}\cdot \nu\df S\df t+\f{1}{2}\iint_Q |\pt_t\vp|^2\Div\bs{n}\df x\df t, 
		\end{split}
	\end{equation*}
	 Given that  $\pt_t\vp=0$ on $\pt Q$,  $\supp \bs{n}\s \mcO(\pt\Om;2R_0)$ and \eqref{12.01.2}, we derive 
	\begin{equation}\label{09.05.1}
		\begin{split}
			|I_1|
			&\leq C\left(\|\vp^0\|_{H_0^1(\Om;w)}^2+\|\vp^1\|_{L^2(\Om)}^2+\|f\|_{L^2(Q)}^2\right),
		\end{split}
	\end{equation}
	where the constant $C>0$ depends only on $\al, R_0, T$ and $\mcO(\pt\Om;2R_0)$.
	
	Further integration by parts, considering $\vp=0$ on $\pt Q$ and $\supp \bs{n}\s \mcO(\pt\Om;2R_0)$, yields 
	\begin{equation*}
		\begin{split}
			I_2
			&=-\iint_{\pt Q} w\left|\f{\pt \vp}{\pt \nu}\right|^2\df S\df t+\iint_Q w\nabla \vp\cdot \nabla (\bs{n}\cdot \nabla \vp)\df x\df t\\
			&=-\f{1}{2}\iint_{\pt Q} w\left|\f{\pt \vp}{\pt \nu}\right|^2\df S\df t+\iint_Q w \left(\nabla \vp\cdot [(D\bs{n})\nabla \vp]\right) \df x\df t\\
			&\hspace{4.5mm}-\f{1}{2}\iint_Q (\bs{n}\cdot \nabla w)|\nabla \vp|^2\df x\df t-\f{1}{2}\iint_Q w|\nabla \vp|^2\Div \bs{n}\df x\df t.
		\end{split}
	\end{equation*}
Combining this result with \eqref{02.13.1} and  \eqref{09.05.1} and $I=I_1+I_2$, we obtain
	\begin{equation}\label{12.02.1}
		\begin{split}
			\iint_{\pt Q} \left|\f{\pt \vp}{\pt \nu}\right|^2\df S\df t
			&\leq C\left(\|\vp^0\|_{H_0^1(\Om;w)}+\|\vp^1\|_{L^2(\Om)}+\|f\|_{L^2(Q)}\right)^2,
		\end{split}
	\end{equation}
	where the constant $C>0$ depends only on $\al, R_0, T$ and $\mcO(\pt\Om;3R_0)$.
	
	{\it Step 8: We prove the forth term in \eqref{11.18.2}.}
	
	Since $z=\pt_t\vp$  serves as the weak solution to the following system: 
	\begin{equation*}
		\begin{cases}
			\pt_{tt}z+\mcA z=\pt_tf, &\mbox{in }Q, \\
			z=0, &\mbox{on }\pt Q, \\
			z(0)=\vp^1, \pt_tz(0)=f(0)-\mcA \vp^0, &\mbox{in }\Om,
		\end{cases}
	\end{equation*}
	 as derived in {\it Step 7}, we obtain 
	\begin{equation*}
		\left\|\f{\pt z}{\pt \nu}\right\|_{L^2(\pt Q)}\leq C\left(\|\vp^0\|_{D(\mcA)}+\|\vp^1\|_{H_0^1(\Om;w)}+\|f\|_{H^1(0,T; L^2(\Om))}\right),
	\end{equation*}
	where the constant $C>0$ depends only on $\al, R_0, T$ and $\mcO(\pt\Om;3R_0)$.

 By combining this result with the fourth term of \eqref{11.18.1} and utilizing the identity
 $\pt_t\f{\pt\vp}{\pt\nu}=\f{\pt (\pt_t \vp)}{\pt\nu}$,  we arrive at the fourth term of \eqref{11.18.2}.
	
	{\it Step 9: We end the proof of this theorem.}

From {\it Steps 1-5}, we establish the first three terms in \eqref{11.18.1}. The fourth term of \eqref{11.18.1} is derived from {\it Step 7}. Additionally, \eqref{11.18.2} is obtained by combining the results from {\it Step 6} and {\it Step 8}. This completes the proof of the theorem.
\end{proof}

 The subsequent lemma provides an equivalent criterion for the existence of solutions to equation \eqref{12.14.1}.

\begin{lemma}\label{11.18.L1}
	The function
	\begin{equation*}
		\vp\in L^2(0,T; H_0^1(\Om;w)), \mbox{ satisfying  } \pt_t\vp\in L^2(Q), \ \pt_{tt}\vp\in L^2(0,T; H^{-1}(\Om;w))
	\end{equation*}
	constitutes the weak solution of \eqref{12.14.1} with respect to $(\vp^0,\vp^1,f)$ if and only if
	\begin{equation*}
		\vp\in L^2(0,T; H_0^1(\Om;w)), \mbox{ with } \pt_t\vp \in  L^2(Q)
	\end{equation*}
	 fulfills the following condition:
	\begin{equation}\label{11.18.3}
		\begin{split}
			&\iint_Q \vp(t)  \pt_{tt}\psi(t) \df x\df t+\iint_Q w\nabla \vp\cdot \nabla \psi\df x\df t\\
			&=\iint_Q f\psi\df x\df t+\int_\Om \vp^1\psi(0)\df x-\int_\Om \vp^0\pt_t\psi(0)\df x
		\end{split}
	\end{equation}
	for all $\psi\in C^\iy(\ol Q)$ such that  $\supp\psi(t)\s \Om$ for all $t\in [0,T]$ and $\psi(T)=\pt_t\psi(T)=0$.
\end{lemma}

\begin{proof}
Let $\psi\in C^\iy(\ol Q)$ with $\supp \psi(t)\s \Om$ for all $t\in [0,T]$ and $\psi(T)=\pt_t\psi(T)=0$. According to Definition \ref{11.13.D1} (i), we have
\begin{equation}\label{02.04.9}
	\begin{split}
		\lg \pt_{tt}\vp, \psi\rg_{H^{-1}(\Om;w), H_0^1(\Om;w)}+\int_\Om w\nabla \vp\cdot\nabla\psi\df x=\int_\Om f\psi\df x
	\end{split}
\end{equation}
for almost every $t\in [0,T]$, and
\begin{equation}\label{02.04.10}
	\vp\in C([0,T]; L^2(\Om)), \text{ and } \pt_t\vp\in C([0,T]; H^{-1}(\Om;w)).
\end{equation} By integrating over the interval $[0,T]$ based on \eqref{02.04.9}, we obtain
\begin{equation*}
	\int_0^T\lg \pt_{tt}\vp, \psi\rg_{H^{-1}(\Om;w), H_0^1(\Om;w)}\df t+\iint_Q w\nabla\vp\cdot \nabla\psi\df x\df t=\iint_Q f\psi\df x\df t.
\end{equation*}
Applying integration by parts with respect to $t\in [0,T]$ and utilizing \eqref{02.04.10}, we derive
\begin{equation*}
	\begin{split}
		&\int_0^T\vp\pt_{tt}\psi\df x\df t+\iint_Q w\nabla\vp\cdot \nabla\psi\df x\df t\\
		&=\iint_Q f\psi\df x\df t-(\vp(0), \pt_t\psi(0))_{L^2(\Om)}+\lg \pt_t\vp(0), \psi(0)\rg_{H^{-1}(\Om;w), H_0^1(\Om;w)}\\
		&=\iint_Q f\psi\df x\df t+\int_\Om \vp^1\psi(0)\df x-\int_\Om \vp^0\pt_t\psi(0)\df x.
	\end{split}
\end{equation*}
This establishes \eqref{11.18.3}.

	  We proceed to establish the sufficiency. 
	
Indeed,   for each $\psi\in L^2(0,T; H_0^1(\Om;w))$,  we have 
	\begin{equation*}
		\begin{split}
			\|\mcA \vp\|_{L^2(0,T; H^{-1}(\Om;w))}
			&=\sup_{  \|\psi\|_{L^2(0,T; H_0^1(\Om;w))}\leq 1}\int_0^T\lg \mcA \vp, \psi\rg_{H^{-1}(\Om;w), H_0^1(\Om;w)}\df t\\
			&=\sup_{  \|\psi\|_{L^2(0,T; H_0^1(\Om;w))}\leq 1}\iint_Q w\nabla \vp\cdot \nabla\psi\df x\df t\leq \|\vp\|_{L^2(0,T; H_0^1(\Om;w))}.
		\end{split}
	\end{equation*}
	This implies that  $\mcA \vp\in L^2(0,T; H^{-1}(\Om;w))$. Now, for each $\psi\in L^2(0,T; H_0^1(\Om;w))$,  since
	\begin{equation*}
		\begin{split}
			\iint_Q \lg f, \xi\rg_{H^{-1}(\Om;w), H_0^1(\Om;w)}\df x\df t
			&=
			\iint_Q f\xi\df x\df t\\
			&\leq \|f\|_{L^2(Q)}\|\xi\|_{L^2(Q)}\leq C\|f\|_{L^2(Q)}\|\xi\|_{L^2(0,T; H_0^1(\Om;w))}
		\end{split}
	\end{equation*}
	 by Remark \ref{08.16.R1}, we conclude that  $f\in L^2(0,T; H^{-1}(\Om;w))$.

Next, taking  $\xi\in C_0^\iy(Q)$, from \eqref{11.18.3},  we obtain 
	\begin{equation*}
		\begin{split}
			\iint_Q \vp\pt_{tt}\xi\df x\df t
			&=\iint_Q f\xi\df x\df t-\iint_Q w\nabla \vp\cdot \nabla \xi\df x\df t\\
			&=\int_0^T\lg f, \xi\rg_{H^{-1}(\Om;w),H_0^1(\Om;w)}\df x\df t-\int_0^T \lg \mcA \vp, \xi\rg_{H^{-1}(\Om;w), H_0^1(\Om;w)}\df t,
		\end{split}
	\end{equation*}
	and thus 
	\begin{equation*}
		\pt_{tt}\vp=f-\mcA \vp
	\end{equation*}
	in the sense of distribution. Therefore,
	\begin{equation*}
		\pt_{tt}\vp\in L^2(0,T; H^{-1}(\Om;w)).
	\end{equation*}
	
	Fix $t\in (0,T)$ and $v\in C_0^\iy(\Om)$ and  $\de\in (0, \f{1}{4}\min\{t,T-t\})$. For any  $\e\in (0,\de)$,  choose  $\zeta\in C^\iy(\ol\R)$
with $0\leq \zeta\leq 1$ such that  $\zeta=1$ on $(t-\de,t+\de)$ and $\supp \zeta\s (t-\de-\e,t+\de+\e)$. Taking $\psi=\zeta v$, then $\psi\in C_0^\iy(Q)$, and from \eqref{11.18.3},  we get (in distribution sense)
	\begin{equation*}
		\begin{split}
			\int_0^T \zeta \lg \pt_{tt}\vp, v\rg_{H^{-1}(\Om;w), H_0^1(\Om;w)}\df t+\iint_Q \zeta w\nabla \vp \cdot\nabla v\df x\df t=\iint_Q \zeta f v\df x\df t.
		\end{split}
	\end{equation*}
	Hence, letting $\e\ra 0$, we obtain 
	\begin{equation*}
		\int_{t-\de}^{t+\de}\lg \pt_{tt}\vp, v\rg_{H^{-1}(\Om;w), H_0^1(\Om;w)}\df t+\iint_{\Om\ts (t-\de,t+\de)} w\nabla\vp \cdot  \nabla v\df x\df t=\iint_{\Om\ts (t-\de,t+\de)} f v\df x\df t.
	\end{equation*}
	Therefore, letting $\de\ra 0$,   for each   $v\in C_0^\iy(\Om)$ and almost every  $t\in (0,T)$, we get
	\begin{equation*}
		\lg \pt_{tt}\vp, v\rg_{H^{-1}(\Om;w), H_0^1(\Om;w)}+\int_\Om w\nabla \vp\cdot \nabla v\df x=\int_\Om fv\df x
	\end{equation*}  
according to the Lebesgue point theorem.   This implies \eqref{11.18.4}  by the density of $C_0^\iy(\Om)$  in $H_0^1(\Om;w)$.
	
	Finally,  denote $\vp(x,t)=\sum_{n=1}^\iy \vp_n(t)\Phi_n(x)$. Then,  for each $\zeta\in C^\iy(\R)$ such that $\zeta(T)=0, \pt_t\zeta(T)=0$, and for each $v\in C_0^\iy(\Om)$, take $\psi(x,t)=\zeta(t)v(x)$. From \eqref{11.18.3} we have
	\begin{equation*}
		\begin{split}
		&\int_0^T (\vp(t),v)_{L^2(\Om)} \pt_{tt}\zeta \df t+\int_0^T (\vp(t),v)_{H_0^1(\Om;w)} \zeta \df t\\ &=\int_0^T (f,v)_{L^2(\Om)}\zeta \df t+\zeta(0)(\vp^1,v)_{L^2(\Om)}-(\pt_t\zeta(0))(\vp^0,v)_{L^2(\Om)}.
		\end{split}
	\end{equation*}
	Letting $v\ra \Phi_n\ (n\in\N)$ strongly in $H_0^1(\Om;w)$, noting that $C_0^\iy(\Om)$ is dense in $H_0^1(\Om;w)$, we obtain 
	\begin{equation*}
		\begin{split}
			\int_0^T\vp_n(t)\pt_{tt}\zeta\df t+\la_n\int_0^T \vp_n(t) \zeta \df t=\int_0^Tf_n(t) \zeta\df t+\zeta(0)\vp_n^0-(\pt_t\zeta(0))\vp_n^1,
		\end{split}
	\end{equation*}
	where $f_n(t)=(f(t), \Phi_n)_{L^2(\Om)}, \vp_n^1=(\vp^1, \Phi_n)_{L^2(\Om)}, \vp_n^0=(\vp^0, \Phi_n)_{L^2(\Om)}$.  
It is easily verified that $\vp_n$ is a weak solution of the following one-dimensional system
	\begin{equation}\label{11.18.6}
		\begin{cases}
			\pt_{tt}\vp_n(t)+\la_n\vp_n(t)=f_n(t), & t\in [0,T],\\
			\vp_n(0)=\vp_n^0, \pt_t\vp_n(0)=\vp_n^1.
		\end{cases}
	\end{equation}
	 By the same argument as in the proof of Theorem \ref{11.14.T1}, we know that the sequence 
	$\{\varphi^m = \sum_{n=1}^m \varphi_n(t)\Phi_n\}_{m\in\mathbb{N}}$
	converges weakly to the solution $\wt\vp$ of \eqref{12.14.1} with respect to the data $(\varphi^0,\varphi^1,f)$. Note that  the sequence $\{\varphi^m\}_{m\in\mathbb{N}}$ converges strongly to $\varphi(x,t)$ in $L^2(Q)$. Hence $\wt\vp=\vp$.  This implies that $\vp(0)=\sum_{n=1}^\iy \vp_n(0)\Phi_n=\vp^0, \pt_t\vp(0)=\sum_{n=1}^\iy [\pt_t\vp_n(0)]\Phi_n=\vp^1$. This completes the proof of the lemma.
\end{proof}
	
	\section{Approximations}\label{S3}
	
	 In this section, we construct a sequence of uniformly hyperbolic equations that serve as approximations to the degenerate hyperbolic equation \eqref{12.14.1}.

We set  $\e\in (0,R_0)$  throughout this section.
	
	\begin{lemma}\label{12.17.L1}
		There exists a $C^{2,1}$ function $\psi_\e: \mathbb{R}\to\mathbb{R}$ with $\e\in (0,R_0)$ satisfying 
		\begin{equation*}
			\psi_\e(x)=|x| \quad \text{for } |x|\geq \e, \quad |x|\leq \psi_\e(x)\leq \e \quad \text{on } [-\e, \e],
		\end{equation*}
		and
		\begin{equation}\label{12.16.1}
			\psi_\e\geq \frac{3\e}{8} \quad \text{on } \mathbb{R}, \quad |\psi_\e'|\leq C   \text{ on } \mathbb{R}, \quad |\psi_\e''|\leq \frac{C}{\e},  \text{ and } |\psi_\e'''|\leq \frac{C}{\e^2}  \text{ on } \mathbb{R},
		\end{equation}
		where the constants $C>0$ are absolute.
	\end{lemma}
	\begin{proof}
		 Consider the polynomial function 
		\begin{equation*}
			\psi_\e(x)=\sum_{i=0}^4a_ix^{4-i},
		\end{equation*}
		where $a_0, \cdots, a_4$  are constants to be determined. We impose the following conditions:
		\begin{equation*}
			\begin{split}
				\psi_\e(-\e)=\e, \quad \psi_\e(\e)=\e, \quad \psi_\e'(-\e)=-1, \quad \psi_\e'(\e)=1, \quad \psi_\e''(\pm \e)=0. 
			\end{split}
		\end{equation*}
		 These conditions yield the solution 
		\begin{equation}\label{12.20.1}
			\psi_\e(x)=\frac{3\e}{8}+\frac{3}{4\e}x^2-\frac{1}{8\e^3}x^4, \quad x\in[-\e,\e].
		\end{equation}
		This function satisfies the required properties.
	\end{proof}

	\begin{remark}\label{06.07.R1}
		Let $x\in\mathbb{R}^N$, and consider the function $\psi_\e(x)=\psi_\e(|x|)$ defined in Lemma  \ref{12.17.L1}. We have (for all $i,j,k=1,\cdots, N$)
		\begin{equation}\label{12.18.1}
			\begin{split}
				\f{\pt \psi_\e}{\pt x_i}
				&=\frac{3}{2\e}x_i-\frac{1}{2\e^3}|x|^2x_i, \quad 
				\frac{\partial^2 \psi_\e}{\partial x_i\partial x_j}
				=\delta_{ij}\left[\frac{3}{2\e}-\frac{1}{2\e^3}|x|^2\right]-\frac{1}{\e^3}x_ix_j
			\end{split}
		\end{equation}
		on $B_\e$,  and
		\begin{equation}\label{12.18.2}
			\begin{split}
				&\f{\pt \psi_\e}{\pt x_i}=|x|^{-1}x_i, \quad \frac{\partial^2\psi_\e}{\partial x_i\partial x_j}=\delta_{ij}|x|^{-1}-x_ix_j|x|^{-3}
			\end{split}
		\end{equation}
		on $\Omega\setminus B_\e$. Here and in what follows, we denote $\de_{ij}, \de^{ij}$ and $\de_i^j$  
denote the Kronecker delta symbol, which equals  $1$ when $i=j$ and   $0$ when $i\neq j$.		
	\end{remark}
	
	Denote
	\begin{equation}\label{11.14.1}
		\begin{split}
			w_\e(x)= (\psi_\e(x))^\al=(\psi_\e(|x|))^\al, \quad x\in\Omega.
		\end{split}
	\end{equation}
	 By direct computation, we obtain
	\begin{equation}\label{11.25.1}
		\f{\pt w_\e}{\pt x_i}=\al \psi_\e^{\al-1}\f{\pt \psi_\e}{\pt x_i}=\al w_\e \psi_\e^{-1}\f{\pt\psi_\e}{\pt x_i}, \mbox{ and } \f{\pt w_\e^{-1}}{\pt x_i}=-\al \psi_\e^{-\al-1}\f{\pt\psi_\e}{\pt x_i}=-\al w_\e^{-1}\psi_\e^{-1}\f{\pt\psi_\e}{\pt x_i}.
	\end{equation}

For each $\e\in (0,R_0)$,  replace$w$  with  $w_\e$.  By employing the same reasoning as in Section \ref{S2}, we can define the following Hilbert spaces:
\begin{equation*}
	H_0^1(\Om;w_\e), \  H^1(\Om;w_\e), \mbox{ and } H^2(\Om;w).
\end{equation*}
 It is worth noting that  $H_0^1(\Om;w_\e)=H_0^1(\Om)$, $H^1(\Om;w_\e)=H^1(\Om)$ and $H^2(\Om;w_\e)=H^2(\Om)$,  which implies $H^{-1}(\Om;w_\e)=H^{-1}(\Om)$.
  Denote
\begin{equation}\label{02.13.2}
	\mcA_\e u=-\Div(w_\e\nabla u), \mbox{ with } D(\mcA_\e)=H^2(\Om;w_\e)\cap H_0^1(\Om;w_\e)=H^2(\Om)\cap H_0^1(\Om).
\end{equation}
	
	\begin{lemma}\label{08.16.L2}
		Let $z\in H_0^1(\Omega;w_\e)$. Then,
		\begin{equation}\label{08.20.1}
			(N+\al-2)^2\int_\Om \psi_\e^{\al-2}z^2\df x\leq 4\int_\Om \psi_\e^\al |\nabla z|^2\df x.
		\end{equation}
		Furthermore, 
		\begin{equation}\label{08.31.10}
			\int_\Om z^2\df x\leq \f{4M^{2-\al}}{(N+\al-2)^2}\int_\Om \psi_\e^\al |\nabla z|^2\df x.
		\end{equation}
	\end{lemma}
	\begin{proof}
		Since $w_\e\in C^{2,1}(\overline{\Omega})$, we have
		\begin{equation*}
			\begin{split}
				2\int_{\Omega} \psi_\e^{\al-2}z(x\cdot\nabla z)\df x
				&=\int_\Omega \psi_\e^{\al-2}x\cdot \nabla z^2\df x=\int_\Omega \Div\left(\psi_\e^{\al-2}z^2x\right)\df x-\int_\Omega z^2\Div\left(\psi_\e^{\al-2}x\right)\df x\\
				&=-(N+\al-2)\int_\Omega z^2\psi_\e^{\al-2} \df x-(2-\al)\int_{B_\e}z^2\psi_\e^{\al-3}[\psi_\e-x\cdot\nabla\psi_\e]\df x.
			\end{split}
		\end{equation*}
		Then,
		\begin{equation*}
			\begin{split}
				(N+\al-2)\int_\Omega z^2\psi_\e^{\al-2}\df x
				&\leq -2\int_\Omega \psi_\e^{\al-2}z(x\cdot \nabla z)\df x=-2\int_\Omega \left(\psi_\e^\f{\al-2}{2}z\right)\left(\psi_\e^\f{\al-2}{2}x\cdot\nabla z\right)\df x\\
				&\leq 2\left(\int_\Omega \psi_\e^{\al-2}z^2\df x\right)^\frac{1}{2}\left(\int_\Omega \psi_\e^{\al-2}|x|^2|\nabla z|^2\df x\right)^\frac{1}{2}\\
				&\leq 2\left(\int_\Omega \psi_\e^{\al-2}z^2\df x\right)^\frac{1}{2}\left(\int_\Omega |\nabla z|^2\psi_\e^\al\df x\right)^\frac{1}{2}
			\end{split}
		\end{equation*}
		due to $\al\in (0,2)$, and $\psi_\e-x\cdot\nabla\psi_\e=\frac{3}{8\e^3}(\e^2-|x|^2)^2\geq 0$ on $B_\e$   (see Remark \ref{06.07.R1}), along with Lemma \ref{12.17.L1}. This establishes \eqref{08.20.1}.

Finally, by utilizing \eqref{08.20.1} and the fact that $(\frac{3\varepsilon}{8})^\alpha\leq w_\varepsilon\leq M^\alpha$ in $\Omega$, we derive \eqref{08.31.10}. This completes the proof of this lemma.
	\end{proof}

The following lemma bears resemblance to Lemma \ref{08.15.L4}; indeed,
 $H_0^1(\Om;w_\e)=H_0^1(\Om)$. 
	
\begin{lemma}\label{11.27.L2}
	Let $\e>0$. The embedding $H_0^1(\Om;w_\e)\hra L^2(\Om)$ is compact.
\end{lemma}

\begin{notation}\label{02.05.N1}
	Let  $\e\in (0,R_0)$ and $\mcA_\e$ be defined in \eqref{02.13.2}.  
From Lemma \ref{08.16.L2}, it follows that the degenerate partial differential operator
 $\mcA_\e$  possesses a discrete point spectrum:
\begin{equation}\label{11.27.5}
	0<\la_1^\e\leq \la_2^\e\leq \la_3^\e\leq  \cdots \ra +\iy,
\end{equation}
namely,  $\la_n^\e\ (n\in\N)$ satisfies the following equation
\begin{equation*}
	\begin{cases}
		\mcA_\e \Phi_n^\e=\la_n^\e \Phi_n^\e, & \mbox{in }\Om, \\
		\Phi_n^\e=0, &\mbox{in }\pt\Om.
	\end{cases}
\end{equation*}
  Hereafter, $\Phi_n^\e(x)$ denotes  the $n$-th eigenfunction of $\mcA_\e$  corresponding to the eigenvalue
   $\la_n^\e\ (n\in\N)$. The set $\{\Phi_n^\e\}_{n\in\N}$  forms an orthonormal basis for $L^2(\Om)$ and is an orthogonal subset of
      $H_0^1(\Om;w_\e)$. (\cite[Theorem 7, Appendix D, p.728]{Evans}.)
\end{notation}

\begin{lemma}\label{02.13.L1}
	Let $u=\sum_{i=1}^\iy u_i \Phi_i^\e\in H_0^1(\Om;w_\e)$ with $u_i=(u,\Phi_i^\e)_{L^2(\Om)}$ for all $i\in\N$.
	Then, $\nabla u=\sum_{i=1}^\iy u_i\nabla \Phi_i^\e$ and $\|u\|_{H_0^1(\Om;w_\e)}=(\sum_{i=1}^\iy u_i^2\la_i^\e)^\f{1}{2}$,  Moreover,
	\begin{equation*}
		u\in H^2(\Om;w_\e)\Lra \sum_{i=1}^\iy u_i^2(\la_i^\e)^2<\iy,
	\end{equation*}
	and
	\begin{equation*}
		\mcA_\e u=\sum_{i=1}^\iy u_i\la_i^\e\Phi_i^\e \mbox{ and } \|\mcA_\e  u\|_{L^2(\Om)}=\left(\sum_{i=1}^\iy u_i^2(\la_i^\e)^2\right)^\f{1}{2}.
	\end{equation*}
\end{lemma}

\begin{proof}
	The proof of this theorem follows the same line of reasoning as that of Lemma  \ref{06.28.L1}.
\end{proof}

 Now, consider the following equation
\begin{equation}\label{11.14.2}
	\begin{cases}
		\pt_{tt}\vp_\e-\Div (w_\e\nabla \vp_\e)=f_\e, &\mbox{in } Q, \\
		\vp_\e=0, &\mbox{on }\pt Q, \\
		\vp_\e=\vp_\e^0, \pt_t\vp(0)=\vp_\e^1, &\mbox{in }\Om,
	\end{cases}
\end{equation}
where $w_\e$ is defined in \eqref{11.14.1}, $\vp_\e^0\in H_0^1(\Om;w_\e)$, $\vp_\e^1\in L^2(\Om)$ and  $f_\e\in L^2(Q)$.
	
\begin{definition}\label{11.14.D1}
	 A function 
	\begin{equation*}
		\vp_\e\in L^2(0,T; H_0^1(\Om;w_\e)), \mbox{ with } \pt_t\vp_\e\in L^2(Q), \ \pt_{tt}\vp_\e\in L^2(0,T; H^{-1}(\Om;w_\e)),
	\end{equation*}
	is said to be a weak solution of \eqref{12.14.1} with respect to $(\vp_\e^0,\vp_\e^1,f_\e)$ provided
	
	(i) for each $\psi_\e\in H_0^1(\Om;w_\e)$ and almost all  $t\in [0,T]$,  the following holds 
	\begin{equation*}
		\lg \pt_{tt}\vp_\e, \psi_\e\rg_{H^{-1}(\Om;w_\e), H_0^1(\Om;w_\e)}+\int_\Om w_\e\nabla \vp_\e \cdot \nabla \psi_\e \df x=\int_\Om f_\e\psi_\e\df x,
	\end{equation*}
	
	(ii) $\vp_\e(0)=\vp_\e^0, \pt_t\vp_\e(0)=\vp_\e^1$.
\end{definition}

 Consequently, we present the following theorem. 
 
\begin{theorem}\label{11.14.T2}
	Let $\e\in (0,R_0)$ be a fixed constant. Given $\vp_\e^0\in H_0^1(\Om;w_\e), \vp_\e^1\in L^2(\Om)$ and $f_\e\in L^2(Q)$,
there exists a unique weak  solution
	\begin{equation}\label{11.14.9}
		\vp_\e\in L^2(0,T; H_0^1(\Om;w_\e)), \mbox{ with } \pt_t\vp_\e\in  L^2(0,T; L^2(\Om)), \ \pt_{tt}\vp\in  L^2(0,T; H^{-1}(\Om;w_\e)).
	\end{equation}
	 Furthermore,
	\begin{equation}\label{12.02.2}
		\f{\pt\vp_\e}{\pt \nu}\in L^2(\pt Q).
	\end{equation}
	 dditionally, the following estimate holds
	\begin{equation}\label{11.14.8}
		\begin{split}
		&\esssup_{t\in [0,T]}\left(\|\vp_\e(t)\|_{H_0^1(\Om;w_\e)}+\|\pt_t\vp_\e(t)\|_{L^2(\Om)}\right)+\|\vp_\e\|_{L^2(0,T; H^{-1}(\Om;w_\e))}+\left\|\f{\pt \vp_\e}{\pt \nu}\right\|_{L^2(\pt Q)}\\
		&\leq C\left(\|\vp_\e^0\|_{H_0^1(\Om;w_\e)}+\|\vp_\e^1\|_{L^2(\Om)}+\|f_\e\|_{L^2(Q)}\right),
		\end{split}
	\end{equation}
	where the constant $C>0$ depends only on $\al, T$ and $\mcO(\pt\Om;3R_0)$.
	
	Moreover, if $f_\e\in H^1(0,T; L^2(\Om)), \vp_\e^0\in D(\mcA_\e), \vp_\e^1\in H_0^1(\Om;w_\e)$, then 
	\begin{equation}\label{11.14.7}
		\vp_\e\in L^2(0,T;D(\mcA_\e)), \mbox{ with } \pt_t\vp_\e \in L^2(0,T; H_0^1(\Om;w_\e)), \ \pt_{tt}\vp_\e\in L^2(Q),
	\end{equation}
	 and the following estimate is valid 
	\begin{equation}\label{11.14.10}
		\begin{split}
			&\esssup_{t\in [0,T]}\left(\|\vp_\e(t)\|_{D(\mcA_\e)}+\|\vp_\e(t)\|_{H^2(\mcO(\pt\Om;2R_0))}+\|\pt_t\vp_\e(t)\|_{H_0^1(\Om;w_\e)}\right)\\
			&\hspace{4.5mm}+\left\|\f{\pt \vp_\e}{\pt \nu}\right\|_{H^1(0,T; L^2(\pt\Om))}\\
			&\leq C\left(\|\vp_\e^0\|_{D(\mcA_\e)}+\|\vp_\e^1\|_{H_0^1(\Om;w_\e)}+\|f_\e\|_{H^1(0,T;L^2(\Om))}\right),
		\end{split}
	\end{equation}
	where the constant $C>0$ depends only on $\al, R_0, T$ and $\mcO(\pt\Om;3R_0)$.
\end{theorem}

\begin{proof}
	 The proof of this theorem follows the same line of reasoning as that of Theorem \ref{11.14.T1}.  
	 
	 {{We observe that all the estimates in Theorem \ref{11.14.T1}, namely \eqref{11.30.7}, \eqref{11.30.8}, \eqref{12.01.2}, \eqref{02.13.3}, \eqref{09.07.16}, and \eqref{12.02.1}, rely only on Gronwall’s inequality, Remark \ref{08.16.R1}, Lemma \ref{06.28.L1}, \cite[Theorem 2 (iii), Chapter 5.9.2, p. 302]{Evans}, and the geometric condition $\mathcal{O}(\partial\Omega;3R_0)$.
	 
	 For each $\varepsilon \in (0,R_0)$, Remark \ref{08.16.R1} is replaced by Lemma \ref{08.16.L2}, and Lemma \ref{06.28.L1} is replaced by Lemma \ref{02.13.L1}. All these ingredients are independent of $\varepsilon \in (0,R_0)$. This proves Theorem \ref{11.14.T2}. }}
\end{proof}

\begin{corollary}\label{11.21.C1}
	The function
	\begin{equation*}
		\vp_\e\in L^2(0,T; H_0^1(\Om;w_\e)), \mbox{ which satisfies
 } \pt_t\vp_\e\in L^2(Q), \ \pt_{tt}\vp_\e\in L^2(0,T; H^{-1}(\Om;w_\e))
	\end{equation*}
	is the weak solution of \eqref{11.14.2}  corresponding to the initial data
  $(\vp_\e^0,\vp_\e^1, f_\e)$ if and only if
	\begin{equation*}
		\vp_\e\in L^2(0,T; H_0^1(\Om;w_\e)), \mbox{ with } \pt_t\vp_\e\in L^2(Q)
	\end{equation*}
	 fulfills the following condition
	\begin{equation}\label{11.21.1}
		\begin{split}
			&\iint_Q \vp_\e(t)  \pt_{tt}\psi(t) \df x\df t+\iint_Q w_\e\nabla \vp_\e\cdot \nabla \psi\df x\df t\\
			&=\iint_Q f_\e\psi\df x\df t+\int_\Om \vp_\e^1\psi(0)\df x-\int_\Om \vp_\e^0\pt_t\psi(0)\df x
		\end{split}
	\end{equation}
	for all test functions  $\psi\in C^\iy(\ol Q)$ such that  $\supp\psi(t)\s \Om$ for all $t\in [0,T]$ and $\psi(T)=\pt_t\psi(T)=0$.
\end{corollary}

\begin{proof}
	 he proof of this corollary closely parallels that of Lemma \ref{11.18.L1}.
\end{proof}

\begin{lemma}\label{02.05.L1}
	Let $\e\in (0,R_0)$. Then $C_0^\iy(\Om)\s D(\mcA_\e)$, and for each $u\in C_0^\iy(\Om)$ 
	\begin{equation*}
		\int_\Om (\mcA_\e u)^2\df x\leq C,
	\end{equation*}
$C$  is independent of  $\e \in (0,R_0)$, but depends on $u$. 
\end{lemma}

\begin{proof}
	For any $u\in C_0^\iy(\Om)$,  by Lemma \ref{12.01.L1}, Remark \ref{06.07.R1}, and the fact that $w_\e=w$ on $\Om-B(0,\e)$,  we know that
 $u\in H_0^1(\Om;w_\e)$. Then,  
	\begin{equation*}
		\begin{split}
			&\int_\Om (\mcA_\e u)^2\df x\\
			&=\int_\Om \left[\Div(\psi_\e^\al \nabla u)\right]^2\df x\leq 2\int_{B_\e} \left[\Div(\psi_\e^\al \nabla u)\right]^2\df x+2\int_\Om \left[\Div(|x|^\al \nabla u)\right]^2\df x\\
			&\leq 2\int_{B_\e} \left[\psi_\e^\al \De u+\al\psi_\e^{\al-1}\left(\f{3}{2\e}x-\f{1}{2\e^3}|x|^2x\right)\cdot\nabla u\right]^2\df x+2\int_\Om [\Div(|x|^\al \nabla u)]^2\df x\\
			&\leq C\e^{N+2\al}\sup_{x\in\Om}|\De u|^2+C\e^{N+2\al-2}\sup_{x\in\Om}|\nabla u|^2+2\int_\Om [\Div(|x|^\al \nabla u)]^2\df x\leq C<+\iy,
		\end{split}
	\end{equation*}
	where   the constant $C>0$ depends  only on $\al, M$ and $|\Om|$ and $u$.
\end{proof}

 We now prove one of the main results of this paper. i.e., Theorem \ref{11.20.T2}.

\vspace{0.2cm}

\noindent {\bf {Proof of Theorem \ref{11.20.T2}.}} 
	 We will establish this theorem through the following three  steps.

	{\it Step 1: The approximations (weak convergence) in  \eqref{11.21.2}.}

	 Note that for any  $u\in H_0^1(\Om)$,  the following inequalities hold 
	\begin{equation}\label{02.05.1}
		\int_\Om w_\e(\nabla u\cdot \nabla u)\df x\leq M^\al \int_\Om |\nabla u|^2\df x, \quad \int_\Om w(\nabla u\cdot \nabla u)\df x\leq M^\al \int_\Om |\nabla u|^2\df x.
	\end{equation}
	 This implies that
 $u\in H_0^1(\Om;w_\e)$ for all $\e\in (0,R_0)$, and $ u\in H_0^1(\Om;w)$.
	
	 By combining \eqref{11.14.8} with \eqref{02.05.1}, we obtain
	\begin{equation}\label{02.05.2}
		\begin{split}
		\|\vp_\e\|_{L^2(0,T; H_0^1(\Om;w_\e))}+\|\pt_t\vp_\e\|_{L^2(Q)}
		&\leq C\left(\|f\|_{L^2(Q)}+\|\vp^0\|_{H_0^1(\Om;w_\e)}+\|\vp^1\|_{L^2(\Om)}\right)\\
		&\leq C\left(\|f\|_{L^2(Q)}+\|\vp^0\|_{H_0^1(\Om)}+\|\vp^1\|_{L^2(\Om)}\right),
		\end{split}
	\end{equation}
	where the constant $C>0$ depends only on $\al, R_0, T$ and $\mcO(\pt\Om;3R_0)$.   It is worth noting that 
	\begin{equation*}
		\int_\Om (\nabla \vp_\e\cdot \nabla \vp_\e)w\df x\leq \int_\Om (\nabla \vp_\e\cdot\nabla \vp_\e)w_\e\df x
	\end{equation*}
	 due to the fact that  $w\leq w_\e$ for all $\e\in (0,R_0)$.  Consequently, we have 
	\begin{equation*}
	\|\vp_\e\|_{L^2(0,T; H_0^1(\Om;w))}+\|\pt_t\vp_\e\|_{L^2(Q)}\leq C\left(\|f\|_{L^2(Q)}+\|\vp^0\|_{H_0^1(\Om)}+\|\vp^1\|_{L^2(\Om)}\right),
	\end{equation*}
	 where the positive constant $C$  again depends solely on $\al, R_0, T$, and $\mcO(\pt\Om;3R_0)$.
 This leads to the conclusion that there exists a subsequence of $\{\vp_\e\}_{\e>0}$, which we still denote by the same notation, and a function
$\wt\vp\in L^2(0,T; H_0^1(\Om;w))$ with $\pt_t\wt\vp\in L^2(Q)$, such that
	\begin{equation}\label{11.21.3}
		\begin{split}
			\vp_\e
			&\ra \wt \vp \mbox{ weakly in } L^2(0,T; H_0^1(\Om;w)), \\
			\pt_t\vp_\e
			&\ra \pt_t\wt \vp \mbox{ weakly in }L^2(Q).
		\end{split}
	\end{equation}
	
	{\it Step 2: $\wt\vp$ is a solution \eqref{12.14.1} with respect to $(\vp^0,\vp^1,f)$, that is, $\wt\vp=\vp$.}
 
  For this very reason, we shall employ \eqref{11.21.1} to establish \eqref{11.18.3}.

Let $\psi\in C^\iy(\ol\Om)$ be such that  $\supp \psi(t)\s \Om$ for all $t\in [0,T]$ and $\psi(T)=\pt_t\psi(T)=0$.  By letting
$\e\ra 0$ 
 and considering the first convergence in \eqref{11.21.3}, we deduce that the limit of the first term in \eqref{11.21.1} is
	\begin{equation*}
		\iint_Q \vp_\e(t)\pt_{tt}\psi(t)\df x\df t\ra \iint_Q \wt\vp(t)\pt_{tt}\psi(t)\df x\df t,
	\end{equation*}
	 Since the right-hand side of \eqref{11.21.1} remains unchanged for
 $\e>0$, it suffices to demonstrate that
	\begin{equation}\label{11.21.4}
		\iint_Q w_\e\nabla\vp_\e\cdot\nabla\psi\df x\df t\ra \iint_Q w\nabla\wt\vp\cdot \nabla\psi\df x\df t
	\end{equation}
	as $\e\ra 0$  by extracting a subsequence.

Now, utilizing \eqref{11.21.4} and \eqref{11.14.1} (which states
 $w_\e=w$ on $\Om-B_\e$), we obtain
	\begin{equation}\label{11.21.5}
		\begin{split}
		\iint_Q w_\e \nabla\vp_\e\cdot\nabla\psi\df x\df t
		&=\iint_Q w\nabla \vp_\e\cdot \nabla\psi\df x\df t+\iint_{B_\e\ts (0,T)}(w_\e-w)\nabla \vp_\e\cdot \nabla \psi\df x\df t
		\end{split}
	\end{equation}
	for all $\e>0$.  From \eqref{11.21.3}, we have
	\begin{equation*}
		\iint_Q w\nabla\vp_\e\cdot\nabla\psi\df x\df t\ra \iint_Q w\nabla\wt \vp\cdot\nabla\psi\df x\df t
	\end{equation*}
	 $\e\ra 0$,  Combining this with \eqref{11.21.4} and \eqref{11.21.5}, we conclude that it is only necessary to show
	\begin{equation}\label{11.21.6}
		\iint_{B_\e\ts (0,T)} (w_\e-w)\nabla\vp_\e\cdot\nabla \psi\df x\df t\ra 0
	\end{equation}
	as $\e\ra 0$.
	
	 Noting from \eqref{02.05.1} and \eqref{02.05.2}, we derive
	\begin{equation*}
		\begin{split}
			&\left|\iint_{B_\e\ts (0,T)}w_\e \nabla\vp_\e\cdot\nabla\psi\df x\df t\right|\\
			&\leq \left(\iint_Q w_\e\nabla\vp_\e\cdot \nabla\vp_\e\df x\right)^\f{1}{2}\left(\iint_{B_\e\ts (0,T)}w_\e\nabla \psi\cdot \nabla\psi\df x\df t\right)^\f{1}{2}\\
			&\leq C\left(\|\vp^0\|_{H_0^1(\Om)}+\|\vp^1\|_{L^2(\Om)}+\|f\|_{L^2(Q)}\right)\|\nabla\psi\|_{L^\iy(Q)}\sqrt{T}\left(w_\e(B_\e)\right)^\f{1}{2}\\
			&\leq C\left(\|\vp^0\|_{H_0^1(\Om)}+\|\vp^1\|_{L^2(\Om)}+\|f\|_{L^2(Q)}\right)\|\nabla\psi\|_{L^\iy(Q)}\sqrt{T}  \e^\f{N+\al}{2},
		\end{split}
	\end{equation*}
	 where the last inequality employed 
	\begin{equation*}
		w_\e(B_\e)=\int_{B_\e} w_\e \df x\leq   \e^\al |B_\e|\leq C\e^{N+\al}
	\end{equation*}
	 by \eqref{11.21.8},  and the constant $C>0$ depends only on $\al, R_0, T$ and $\mcO(\pt\Om;3R_0)$.

	Similarly, we can show that 
	\begin{equation*}
		\begin{split}
			\left|\iint_{B_\e\ts (0,T)} w\nabla\vp_\e\cdot \nabla\psi\df x\df t\right|
			&\leq C\left(\|\vp^0\|_{H_0^1(\Om)}+\|\vp^1\|_{L^2(\Om)}+\|f\|_{L^2(Q)}\right)\|\nabla\psi\|_{L^\iy(Q)}\sqrt{T}  \e^\f{N+\al}{2},
		\end{split}
	\end{equation*}
	 where we have utilized 
	\begin{equation*}  w(B_\e)=\int_{B_\e}w\df x\leq \e^\al |B_\e|\leq C\e^{N+\al}
	\end{equation*}
	by \eqref{11.21.8} and Lemma \ref{12.17.L1}, and the constant $C>0$ depends only on $\al, R_0, T$ and $\Om$.
	
 Combining these results with \eqref{11.21.6}, we establish \eqref{11.18.3}, thereby confirming that
 $\wt\vp=\vp$. This completes the proof of \eqref{11.21.2}.
	
	{\it Step 3}.   Now, we make the assumptions that
$\vp^0\in C_0^\iy(\Om),\vp^1\in H_0^1(\Om)$, and $f\in H^1(0,T; L^2(\Om))$.
	
From \eqref{11.14.10}, Lemma \ref{12.01.L1}, and Lemma \ref{02.05.L1}, we can derive
	\begin{equation*}
		\begin{split}
		\|\vp_\e\|_{L^2(0,T; H^2(\mcO(\pt\Om;2R_0)))}
		&\leq C\left(\|\vp^0\|_{D(\mcA_\e)}+\|\vp^1\|_{H_0^1(\Om)}+\|f\|_{H^1(0,T; L^2(\Om))}\right)\leq C,
		\end{split}
	\end{equation*}
	where the constant $C>0$ is independent of $\e\in (0,R_0)$.  By applying the standard Sobolev trace theorem, we obtain 
	\begin{equation*}
		\begin{split}
			\left\|\f{\pt \vp_\e}{\pt\nu}\right\|_{L^2(0,T; H^\f{1}{2}(\pt\Om))}\leq C.
		\end{split}
	\end{equation*}
	 Consequently, there exists a function  $z\in L^2(0,T; H^\f{1}{2}(\pt\Om))$ such that
	\begin{equation}\label{02.05.3}
		\begin{split}
		\vp_\e
			&\ra \vp \mbox{ weakly in } L^2(0,T; H^2(\mcO(\pt\Om;2R_0))), \\
		\f{\pt \vp_\e}{\pt\nu}
		&\ra z \mbox{ weakly in } L^2(0,T; H^\f{1}{2}(\pt\Om)).
		\end{split}
	\end{equation}
	For any $\psi\in C^\iy(\pt Q)$, we can extend $\psi$ to a $C^\iy$ function on $\ol\Om$, still denote by $\psi$, such that  $\supp \psi\s \mcO(\pt\Om;2R_0)$. 
Then, utilizing \eqref{11.21.3}, $\wt\vp=\vp$, \eqref{02.05.3}, and
 $w_\e=w$ on $\mcO(\pt\Om;2R_0)$, we  have
	\begin{equation*}
		\begin{split}
			\iint_{\pt Q} w_\e\f{\pt\vp_\e}{\pt\nu}\psi\df S\df t
			&=\iint_{\pt Q} \Div(\psi w_\e\nabla\vp_\e )\df S\df t\\
			&=\iint_Q \Div(w_\e\nabla\vp_\e)\psi\df x\df t+\iint_Q w_\e\nabla\vp_\e\cdot\nabla\psi\df x\df t\\
			&\ra \iint_Q \Div(w\nabla\vp)\psi\df x\df t+\iint_Q w\nabla\vp\cdot \nabla\psi\df x\df t=\iint_{\pt Q} w\f{\pt\vp}{\pt\nu}\psi\df S\df t,
		\end{split}
	\end{equation*}
	hence, $z=\f{\pt\vp}{\pt\nu}$. i.e.,
	\begin{equation}\label{02.05.4}	\f{\pt\vp_\e}{\pt\nu}\ra \f{\pt\vp}{\pt\nu} \mbox{ weakly in }L^2(0,T; H^\f{1}{2}(\pt\Om)).
	\end{equation}
	
	Since   $C_0^\iy(\Om)\s D(\mcA_\e)$ for all $\e\in (0,R_0)$ by Lemma \ref{02.05.L1}, and  $\eta_\e=\pt_t\vp_\e$  satisfies the following equation (it suffices to reprove Theorem \ref{11.14.T1})
 	\begin{equation*}
		\begin{cases}
			\pt_{tt}\eta_\e +\mcA_\e\eta_\e=\pt_t f, &\mbox{in }Q, \\
			\eta_\e=0, &\mbox{on }\pt Q, \\
			\eta_\e(0)=\vp^1, \pt_t\eta_\e(0)=f(0)-\mcA_\e\vp^0, &\mbox{in }\Om,
		\end{cases}
	\end{equation*}
	 then, by Lemma \ref{02.05.L1} and \eqref{11.14.8}, we get
	\begin{equation}\label{02.05.5}
		\left\|\pt_t\f{\pt \vp_\e}{\pt\nu}\right\|_{L^2(\pt Q)}=\left\|\f{\pt\eta_\e}{\pt\nu}\right\|_{L^2(\pt Q)}\leq C<+\iy,
	\end{equation}
	where the constant $C>0$ is independent of $\e\in (0,R_0)$.
	 
Therefore,  the sequence $\{\f{\pt\vp_\e}{\pt\nu}\}_{\e\in (0,R_0)}$
is  bounded in $L^2(0,T; H^\f{1}{2}(\pt\Om))$ by \eqref{02.05.4},  and  $\{\pt_t\f{\pt\vp_\e}{\pt\nu}\}_{\e\in (0,R_0)}$ is bounded in $L^2(\pt Q)$ by \eqref{02.05.5}.   Combining these facts with the compactness of the embeddings   $H^{\frac{1}{2}}(\partial \Omega) \hookrightarrow L^2(\partial \Omega)$ and   
\begin{equation*}
		\left\{\xi\in L^2(0,T; H^\f{1}{2}(\pt\Om))\colon \pt_t\xi \in L^2(\pt Q)\right\}\hra L^2(\pt Q),
	\end{equation*}
	we can establish \eqref{11.21.7} through an abstract subsequence argument. This completes the proof of this theorem.
\hfill $\Box$ 

\begin{remark}
	{{Theorem \ref{11.20.T2} is far from being a trivial consequence. Its proof relies essentially on Theorems \ref{11.14.T1} and \ref{11.14.T2}, whose estimates are uniform with respect to the parameter $\varepsilon \in (0,R_0)$.
	
	A major difficulty in the proof of Theorem \ref{11.20.T2} lies in the fact that the weighted Sobolev spaces (i.e., the solution spaces) depend on $\varepsilon \in (0,R_0)$. This dependence creates substantial technical complications. To overcome this issue, we must embed the solutions $\varphi_\varepsilon$ into a fixed space, namely $H_0^1(\Omega;w)$, in order to establish the convergence stated in \eqref{11.21.3}. Consequently, the choice of the approximating spaces $H_0^1(\Omega;w_\varepsilon)$ is crucial: it determines whether the convergence in \eqref{11.21.3} can be achieved.}}
	
\end{remark}

\section{Observability}\label{S4}
 In this section, our focus is on deriving the observability inequality for equation \eqref{12.14.1}. To achieve this, we initially demonstrate a uniform observability result pertaining to the approximate equation \eqref{11.14.2}. 
 
We set  $\e\in (0,R_0)$  throughout this section.

\subsection{Uniformly observability for  the approximate equations}

Throughout this section, we adhere to the standard notations commonly employed in differential geometry. For instance, we also adopt the notatione $\pxi$ to  represent $\f{\pt}{\pt x_i}$, among others.

We commence by providing a description of the differential operator that acts on functions defined on
$\Omega \subset \mathbb{R}^N$.

\begin{notation}
Let $\R^n$  be endowed with its usual topology, and let $x=(x^1,\cdots, x^N)$ 
 denote the natural coordinate system. We define the standard product and norm on  
  $\R^N$ as follows: 
\begin{equation*}
	X\cdot Y=\lg X, Y\rg_0=\sum_{i=1}^N X^iY^i, \quad |X|_0=(X\cdot X)^\f{1}{2},
\end{equation*}
where $X=\sum_{i=1}^NX^i\pxi, Y=\sum_{i=1}^N Y^i\pxi\in\R_x^N$ and $\R_x^N$  represents the tangent space at the point 
 $x\in\R^N$. For a given  $x\in\R^N$, we define  $A_\e(x)=w_\e(x) I_N$,  whereupon we have
\begin{equation}\label{11.28.2}
	A_\e X=\sum_{i=1}^N w_\e X^i\pxi, \mbox{ for all } X=\sum_{i=1}^N X^i\pxi\in\R_x^N,
\end{equation}
with  $I_N$ denoting the identity matrix of order $N$.

Let $X=\sum_{i=1}^N X^i\pxi$ be a vector field on $\R^N$. We denote the divergence of  $X$ 
 with respect to the Euclidean metric by 
\begin{equation*}
	\Div_0(X)=\sum_{i=1}^N\f{\pt X^i}{\pt x^i}.
\end{equation*}
For each $f\in C^1(\ol\Om)$, we define the gradient of $f$ by
\begin{equation}\label{11.26.2}
	\nabla f=\nabla_0 f=\sum_{i=1}^N \f{\pt f}{\pt x^i} \pxi.
\end{equation}
\end{notation}

 Subsequently, we proceed to define the differential operator on the Riemannian manifold $\Omega$.

\begin{notation}Set
	\begin{equation}\label{11.23.1}
		g_\e(x)=(g_{ij}^\e)_{i,j=1,\cdots, N}=w_\e^{-1}I_N=A_\e(x)^{-1} \mbox{ for all } x\in \Om,
	\end{equation}
	that is,  $g_{ij}^\e=w_\e^{-1}\de_{ij}$ for all $i,j=1,\cdots, N$,
	where  $\de_{ij}$ denotes the Kronecker symbol (refer to  Remark \ref{06.07.R1}). Then,
	\begin{equation}\label{11.23.2}
		g_\e^{-1}=(g_\e^{ij})_{i,j=1,\cdots, N}=w_\e I_N,
	\end{equation}
	namely, $g_\e^{ij}=w_\e \de^{ij}$.  It can be readily verified that 
	\begin{equation}\label{11.22.1}
		\f{\pt g_{ij}^\e}{\pt x^k}=(-\al)\psi_\e^{-\al-1}\f{\pt\psi_\e}{\pt x^k}\de_{ij}, \mbox{ and } \f{\pt g_\e^{ij}}{\pt x^k}=\al \psi_\e^{\al-1}\f{\pt\psi_\e}{\pt x^k}\de^{ij}
	\end{equation}
	 by virtue of \eqref{11.14.1}, \eqref{11.25.1}, and \eqref{11.23.2}. 
	
For each $x\in\R^N$,  define the inner product and norm on the tangent space $\R_x^N=\R^N$ as follows: 
\begin{equation}\label{11.25.7}
	g_\e(X,Y)=\lg X, Y\rg_{g_\e}=\sum_{i,j=1}^N g_{ij}^\e X^iY^j,\quad  |X|_{g_\e}=\lg X, X\rg_{g_\e}^\f{1}{2}=\left(\sum_{i=1}^N w_\e^{-1}(X^i)^2\right)^\f{1}{2},
\end{equation}
where $X=\sum_{i=1}^NX^i\pxi, Y=\sum_{i=1}^NY^i\pxi\in\R_x^N$. Moreover,
\begin{equation}\label{11.24.5}
	\llg \pxi, \pxj\rrg_{g_\e}=w_\e^{-1}\de_{ij} \mbox{ for all } i,j=1,\cdots, N,
\end{equation}
and
\begin{equation}\label{11.28.1}
	\lg X, A_\e(x)Y\rg_{g_\e}=X\cdot Y \mbox{ for all } x\in\R^N.
\end{equation}

   We directly confirm that $(\Om, g_\e)$ constitutes a  $C^2$ Riemannian manifold. Its Christoffel symbol $\Ga_{ij}^k=\f{1}{2}\sum_{l=1}^N g_\e^{kl}(\f{\pt g_{li}^\e}{\pt x^j}+\f{\pt g_{lj}^\e}{\pt x^i}-\f{\pt g_{ij}^\e}{\pt x^l})\ (i,j,k=1,\cdots, N)$ is
\begin{equation}\label{11.24.9}
	\Ga_{ij}^k=-\f{\al}{2}\psi_\e^{-1}\left(\f{\pt\psi_\e}{\pt x^j}\de_{ik}+\f{\pt\psi_\e}{\pt x^i}\de_{jk}-\f{\pt\psi_\e}{\pt x^k}\de_{ij}\right)
\end{equation}
 according to \eqref{11.23.2} and \eqref{11.22.1}. 

 The gradient of a function  $u$ on  $(\Om,g_\e)$ is expressed as 
\begin{equation}\label{11.24.3}
	\nabla_{g_\e} u=\sum_{i=1}^N g_\e^{ij}\f{\pt u}{\pt x^i}\f{\pt}{\pt x^j}=\sum_{i=1}^N w_\e \f{\pt u}{\pt x^i}\pxi.
\end{equation}
 Then, from \eqref{11.24.5}, we obtain
\begin{equation}\label{11.24.4}
	\lg \nabla_{g_\e}u, \nabla_{g_\e}v\rg_{g_\e}=\sum_{i,j=1}^N w_\e^2\f{\pt u}{\pt x^i}\f{\pt v}{\pt x^j}\llg \pxi,\pxj\rrg_{g_\e}=\sum_{i=1}^N w_\e\f{\pt u}{\pt x^i}\f{\pt v}{\pt x^i}=(\nabla_0 u\cdot \nabla_0v)w_\e.
\end{equation}

The divergence of a vector field $X=\sum_{i=1}^N X^i\pxi$ on $(\Om,g_\e)$ is given by
\begin{equation*}
	\begin{split}
		\Div_{g_\e} X=\sum_{i=1}^N\f{\pt X^i}{\pt x^i}+\sum_{i,k=1}^N X^i\Ga_{ik}^k=\sum_{i=1}^N \f{\pt X^i}{\pt x^i}-\f{\al N}{2} \psi_\e^{-1} \sum_{i=1}^N X^i\f{\pt\psi_\e}{\pt x^i}.
	\end{split}
\end{equation*}
Hence
\begin{equation}\label{11.24.1}
	\mcA u=-\sum_{i=1}^N \pxi\left(w_\e(x)\f{\pt u}{\pt x_i}\right)=-\Div_0(\nabla_{g_\e}u).
\end{equation}

Denote the Levi-Civita connection in metric $g_\e$ by $D$, that is, 
\begin{equation}\label{11.24.8}
	D_{\pxi}\pxj=\sum_{k=1}^N\Ga_{ij}^k \pxk \mbox{ for all } i,j=1,\cdots, N.
\end{equation}
In other words,  for each vector field $H$ on Riemannian manifold $(\Om,g_\e)$, we have
\begin{equation}\label{11.24.7}
	DH(X,Y)=\lg D_X H, Y\rg_{g_\e} \mbox{ for all } X,Y\in\R_x^N.
\end{equation}
\end{notation}

\begin{remark}\label{11.24.R1}
	 On the domain  $\Om\s\R^N$,  there exist two distinct metrics: one is the well-known 
standard Euclidean metric, and the other is the Riemannian metric
$g_\e$ (where $0<\e<R_0$ is fixed).

In the subsequent discussion, we use the notation
$\nabla f$ represent $ \nabla_0 f$ (refer to \eqref{11.26.2}), and employ 
$\nabla_{g_\e}$  to denote the gradient on the Riemannian manifold
$(\Omega, g_\e)$. 
\end{remark}

 In the following sections, we will present several crucial conclusions.  We note that equations \eqref{11.14.2} are uniformly hyperbolic for all $\varepsilon \in (0,R_0)$. Consequently, integrations by parts are fully justified.

\begin{lemma}\label{11.24.L1}
	Let $H$ be a $C^1$ vector field defined on the Riemannian  manifold $(\Om,g_\e)$, and let $f\in C^2(\Om)$. Then 
for all $x\in\R^N$, the following equation holds:
	\begin{equation*}
		\begin{split}
			&\lg \nabla_{g_\e} f, \nabla_{g_\e} (H(f))\rg_{g_\e}(x)\\
			&=DH(\nabla_{g_\e}f, \nabla_{g_\e}f)(x)+\f{1}{2}\Div_0\left(|\nabla_{g_\e} f|_{g_\e}^2H\right)(x)-\f{1}{2}|\nabla_{g_\e}f|_{g_\e}^2(x)\Div_0(H)(x). 
		\end{split}
	\end{equation*} 
\end{lemma}

\begin{proof}
	 Suppose
 $H=\sum_{i=1}^N H^i\pxi$. Then $H(f)=\sum_{i=1}^N H^i\f{\pt f}{\pt x^i}$.  By referring to \eqref{11.24.3}, we can obtain
	\begin{equation*}
		\begin{split}
		\nabla_{g_\e}\left(\sum_{i=1}^N H^i\f{\pt f}{\pt x^i}\right)
		&=\sum_{j=1}^N w_\e \pxj\left(\sum_{i=1}^N H^i\f{\pt f}{\pt x^i}\right)\pxj \\
		&=\sum_{i,j=1}^N w_\e \f{\pt H^i}{\pt x^j}\f{\pt f}{\pt x^i}\pxj +\sum_{i,j=1}^N w_\e H^i\f{\pt^2f}{\pt x^j\pt x^i}\pxj.
		\end{split}
	\end{equation*}
	Consequently, 
	\begin{equation}\label{11.24.6}
		\begin{split}
			\llg \nabla_{g_\e} f, \nabla_{g_\e}(H(f))\rrg_{g_\e}
			&=\sum_{i,j=1}^N w_\e\f{\pt H^i}{\pt x^j}\f{\pt f}{\pt x^i}\f{\pt f}{\pt x^j}+\sum_{i,j=1}^N w_\e H^i \f{\pt^2 f}{\pt x^i\pt x^j}\f{\pt f}{\pt x^j}
		\end{split}
	\end{equation}
	according to \eqref{11.24.4} and \eqref{11.24.5}.
	
	  On the one hand, we have the following expression 
	\begin{equation*}
		\begin{split}
			DH(\nabla_{g_\e}f, \nabla_{g_\e}f)
			&=\llg D_{\nabla_{g_\e}f}H, \nabla_{g_\e}f\rrg_{g_\e} =\sum_{i,j=1}^N w_\e \f{\pt H^i}{\pt x^j}\f{\pt f}{\pt x^i}\f{\pt f}{\pt x^j}+\sum_{i,j,k=1}^Nw_\e H^j\f{\pt f}{\pt x^i}\f{\pt f}{\pt x^k}\Ga_{ij}^k\\
			&=\sum_{i,j=1}^N w_\e \f{\pt H^i}{\pt x^j}\f{\pt f}{\pt x^i}\f{\pt f}{\pt x^j}-\f{\al}{2}w_\e\psi_\e^{-1}\sum_{i,j=1}^N H^j \f{\pt\psi_\e}{\pt x^j}\left(\f{\pt f}{\pt x^i}\right)^2.
		\end{split}
	\end{equation*}
	This result is derived by applying equations \eqref{11.24.7}, \eqref{11.24.5}, \eqref{11.24.9}, and \eqref{11.24.8}. Additionally, we note that  
	\begin{equation*}
		\begin{split}
			\sum_{i,j,k=1}^Nw_\e H^j\f{\pt f}{\pt x^i}\f{\pt f}{\pt x^k}\Ga_{ij}^k=-\f{\al}{2}w_\e\psi_\e^{-1}\sum_{i,j=1}^N H^j \f{\pt\psi_\e}{\pt x^j}\left(\f{\pt f}{\pt x^i}\right)^2. 
		\end{split}
	\end{equation*}
 On the other hand, we consider the following expression
	\begin{equation*}
		\begin{split}
			&\f{1}{2}\Div_0\left(|\nabla_{g_\e}f|_{g_\e}^2H\right)-\f{1}{2}|\nabla_{g_\e}f|_{g_\e}^2\Div_0(H)\\
			&=\f{1}{2}\sum_{i=1}^N \left(\pxi \lg \nabla_{g_\e}f, \nabla_{g_\e}f\rg_{g_\e} \right)H^i=\f{1}{2}\sum_{i,j=1}^n \f{\pt w_\e}{\pt x^i}\left(\f{\pt f}{\pt x^j}\right)^2H^i+\sum_{i,j=1}^N w_\e\f{\pt^2f}{\pt x^i\pt x^j}\f{\pt f}{\pt x^j}H^i\\
			&=\f{\al}{2}\sum_{i,j=1}^N w_\e\psi_\e^{-1}H^i\f{\pt\psi_\e}{\pt x^i}\left(\f{\pt f}{\pt x^j}\right)^2+\sum_{i,j=1}^N w_\e\f{\pt^2f}{\pt x^i\pt x^j}\f{\pt f}{\pt x^j}H^i
		\end{split}
	\end{equation*}
	This is obtained by utilizing equations \eqref{11.24.4} and \eqref{11.25.1}.
Combining these results with \eqref{11.24.6}, we successfully complete the proof of the lemma.

\end{proof}

\begin{proposition}\label{11.24.P1}
	Let $\vp$ be a solution of the equation 
	\begin{equation}\label{11.26.1}
		\pt_{tt}\vp+\mcA_\e\vp=0 \mbox{ in }Q.
	\end{equation}
	
	1).   Suppose  $H$ is a $C^1$  vector field defined on  $\ol\Om$.  Then, the following identity holds:
	\begin{equation*}
		\begin{split}
			&\iint_{\pt Q}\f{\pt \vp}{\pt \nu_\e}H(\vp)\df S\df t+\f{1}{2}\iint_{\pt Q} \left((\pt_t\vp)^2-|\nabla_{g_\e}\vp|_{g_\e}^2\right)H\cdot \nu \df S\df t\\
			&=\int_\Om [(\pt_t\vp)H(\vp)](t)\df x\bigg|_{t=0}^{t=T}\\
			&\hspace{4.5mm}+\iint_Q DH(\nabla_{g_\e}\vp, \nabla_{g_\e}\vp)\df x\df t+\f{1}{2}\iint_Q \left((\pt_t\vp)^2-|\nabla_{g_\e}\vp|_{g_\e}^2\right)\Div_0(H)\df x\df t.
		\end{split}
	\end{equation*}
	 Here and in what follows, we denote $\f{\pt \vp}{\pt \nu_\e}=w_\e \nabla \vp\cdot \nu$ for each $\e\in (0,R_0)$.
	
	2).  
Assume $P\in C^2(\ol\Om)$.  Then, we have the following equality:
	\begin{equation*}
		\begin{split}
			\iint_Q P\left((\pt_t\vp)^2-|\nabla_{g_\e}\vp|_{g_\e}^2\right)\df x\df t
			&=\int_\Om (\pt_t\vp)\vp P\df x\bigg|_{t=0}^{t=T}+\f{1}{2}\iint_Q \vp^2 \mcA_\e P\df x\df t\\
			&\hspace{4.5mm}+\f{1}{2}\iint_{\pt Q}\vp^2 \nabla_{g_\e}P\cdot \nu \df S\df t-\iint_{\pt Q}\f{\pt \vp}{\pt \nu_\e}\vp P\df S\df t.
		\end{split}
	\end{equation*}
\end{proposition}

\begin{proof}
	 1). Multiplying both sides of the first inequality in \eqref{11.26.1} by
$H(\vp)$,  integrating by parts, and applying Green's formula, we arrive at
	\begin{equation*}
		\begin{split}
			&\iint_Q (\pt_{tt}\vp)H(\vp)\df x\df t\\
			&=\iint_Q \pt_t\left((\pt_t\vp)H(\vp)\right)\df x\df t-\iint_Q (\pt_t\vp)\pt_t[H(\vp)]\df x\df t\\
			&=\int_\Om (\pt_t\vp(t))[H(\vp)(t)]\df x\bigg|_{t=0}^{t=T}-\f{1}{2}\iint_Q H([\pt_t\vp]^2)\df x\df t\\
			&=\int_\Om [\pt_t\vp(t)][H(\vp)(t)]\df x\bigg|_{t=0}^{t=T}+\f{1}{2}\iint_Q (\pt_t\vp)^2\Div_0(H)\df x\df t-\f{1}{2}\iint_{\pt Q} (\pt_t\vp)^2 H\cdot \nu \df S\df t. 
		\end{split}
	\end{equation*}
	 Furthermore, we have
	\begin{equation*}
		\begin{split}
			\iint_Q (\mcA_\e \vp)H(\vp)\df x\df t
			&=-\iint_{\pt Q}\Div_0[(w_\e\nabla \vp) H(\vp)]\df x\df t+\iint_Q w_\e \nabla\vp\cdot \nabla[H(\vp)]\df x\df t\\
			&=-\iint_{\pt Q}\f{\pt \vp}{\pt \nu_\e}H(\vp)\df S\df t+\iint_Q \lg \nabla_{g_\e}\vp, \nabla_{g_\e}(H(\vp))\rg_{g_\e}\df x\df t\\
			&=-\iint_{\pt Q}\f{\pt \vp}{\pt \nu_\e}H(\vp)\df S\df t+\iint_Q DH(\nabla_{g_\e}\vp, \nabla_{g_\e}\vp)\df x\df t\\
			&\hspace{4.5mm}+\f{1}{2}\iint_Q \Div_0\left(|\nabla_{g_\e}\vp|_{g_\e}^2H\right)\df x\df t-\f{1}{2}\iint_Q |\nabla_{g_\e} \vp|_{g_\e}^2\Div_0(H)\df x\df t\\
			&=-\iint_{\pt Q}\f{\pt \vp}{\pt \nu_\e}H(\vp)\df S\df t+\iint_Q DH(\nabla_{g_\e}\vp, \nabla_{g_\e}\vp)\df x\df t\\
			&\hspace{4.5mm}+\f{1}{2}\iint_{\pt Q} |\nabla_{g_\e}\vp|_{g_\e}^2H\cdot\nu\df S\df t-\f{1}{2}\iint_Q |\nabla_{g_\e} \vp|_{g_\e}^2\Div_0(H)\df x\df t
		\end{split}
	\end{equation*}
	 Utilizing \eqref{11.24.3}, \eqref{11.24.4}, and Lemma \ref{11.24.L1}, these two equalities lead to (1) in this proposition.
	
	2).  It is evident that
	\begin{equation*}
		\mcA_\e P=-\Div_0(w_\e\nabla P)=-\Div_0(\nabla_{g_\e}P)
	\end{equation*}
	according to \eqref{11.24.3}. Observe that from \eqref{11.25.7} and \eqref{11.24.4}, we obtain
	\begin{equation*}
		\begin{split}
			\llg \nabla_{g_\e}\vp, \nabla_{g_\e}(P\vp)\rrg_{g_\e}
			&=\llg \nabla \vp, \nabla (P\vp)\rrg_0 w_\e =\lg \nabla\vp, \vp\nabla P\rg_0 w_\e +\lg \nabla \vp, P\nabla \vp\rg_0 w_\e\\
			&=P|\nabla_{g_\e}\vp|_{g_\e}^2+\f{1}{2}\nabla_{g_\e}P(\vp^2)=P|\nabla_{g_\e} \vp|_{g_\e}^2+\f{1}{2}\Div_0\left(\vp^2\nabla_{g_\e}P\right)+\f{1}{2}\vp^2\mcA_\e P,
		\end{split}
	\end{equation*}
	 Combining this with the divergence theorem, we deduce
	\begin{equation*}
		\begin{split}
			\int_\Om (\pt_t\vp)(\vp P)\df x\bigg|_{t=0}^{t=T}
			&=\iint_Q\left[(\pt_{tt}\vp)(\vp P)+(\pt_t\vp)^2P\right]\df x\df t\\
			&=\iint_Q  \left[(\vp P)\Div_0(w_\e\nabla \vp)+(\pt_t\vp)^2P\right]\df x\df t\\
			&=\iint_{\pt Q}\vp P\f{\pt \vp}{\pt \nu_\e}\df S\df t+\iint_Q \left[-\lg \nabla_{g_\e}\vp, \nabla_{g_\e}(\vp P)\rg_{g_\e}+(\pt_t\vp)^2P\right]\df x\df t\\
			&=\iint_QP\left((\pt_t\vp)^2-|\nabla_{g_\e}\vp|_{g_\e}^2\right)\df x\df t-\f{1}{2}\iint_Q\vp^2\mcA_\e P\df x\df t\\
			&\hspace{4.5mm}+\iint_{\pt Q}\f{\pt \vp}{\pt \nu_\e}\vp P\df S\df t-\f{1}{2}\iint_{\pt Q}\vp^2\nabla_{g_\e}P\cdot \nu\df S\df t.
		\end{split}
	\end{equation*}
	This  proves 2).  
\end{proof}

\begin{lemma}\label{12.04.L1}
	 Consider the function, defined for $(0<\e<\min\{R_0,\f{1}{4}\})$, as follows:
	\begin{equation*}
		f(x)=3\e^4+6\e^2(1-\al)x^2-(1-2\al)x^4 \mbox{ for } x\in [0,\e].
	\end{equation*}
	 This function satisfies the following property
	\begin{equation}\label{02.06.1}
		f(x)\geq \wh a, \mbox{ with }
		\wh a=
		\begin{cases}
			\e^4\min\{3,4(2-\al)\}, &\al\in (0,1], \\
			\e^4\min\left\{3,  \f{3}{2\al-1}(3\al-2)(2-\al)\right\}, &\al\in (1,2).
		\end{cases}
	\end{equation}
\end{lemma}

\begin{proof}
	 First, observe that the derivative of
$f(x)$
 is given by   $f'(x)=4x[3\e^2(1-\al)-(1-2\al)x^2]$.

	For $\alpha \in (0,\tfrac{1}{2})$,  the equation $f'(x) = 0$  is satisfied if and only if 
	$x = 0$ or $x = \epsilon \sqrt{\frac{3(1-\alpha)}{1-2\alpha}}$.
	 Note that within this range of  $\alpha$, we have 
	$\sqrt{\frac{3(1-\alpha)}{1-2\alpha}} > 1$. 
	  Consequently, any maximum or minimum of $f(x)$  must occur at either $x = 0$ or $x = \epsilon$.  Furthermore, it follows that 
	$f(x) \ge \epsilon^{4} \min\{3,\, 4(2-\alpha)\}$.
	
	When $\f{1}{2}\leq \al\leq 1$,  the equation $f'(x)=0$ if and only if $x=0$.  In this case, $f(x)$  
	is monotonic on the interval $[0,\e]$, and thus we have $
	f(x)\geq \e^4\min \{3, 4(2-\al)\}$.
	
	When $\al\in (1,2)$,  the equation $f'(x)=0$ holds if and only if $x=0$ or $x=\e\sqrt{\f{3(1-\al)}{1-2\al}}$.  
 In this range of $\al$, we have $\sqrt{\f{3(1-\al)}{1-2\al}}< 1$. 
$f(x)$
 must be attained at  $x = 0$ or $x = \epsilon$ or $x=\e\sqrt{\f{3(1-\al)}{1-2\al}}$. Additionally, since $4(2-\al)\geq \f{3}{2\al-1}(3\al-2)(2-\al)$, we conclude that 
  $f(x)\geq \e^4\min\{3,\f{3}{2\al-1}(3\al-2)(2-\al)\}$. 
	This completes the proof of the lemma.
\end{proof}

 The subsequent analysis is based on Lemmas \ref{11.24.L1}, \ref{12.04.L1}, and Proposition \ref{11.24.P1}.

\begin{lemma}\label{11.26.L1}
	Let
	\begin{equation}\label{11.25.2}
		H=\sum_{k=1}^N H^k\pxk \mbox{ with } H^k=x^k, k=1,\cdots, N.
	\end{equation}
	Then for all $X\in\R_x^N$ and $x\in\Om$, there holds 
	\begin{equation}\label{11.26.3}
		\begin{split}
			\lg D_X H, X\rg_{g_\e}\geq a |X|_{g_\e}^2,  \mbox{ with } a=\f{3}{8}\min\{1,2-\al\}.
		\end{split}
	\end{equation}
\end{lemma}

\begin{proof}
	Let $X=\sum_{i=1}^N X^i \pxi$.  Observe that the function 
	\begin{equation*}
		\begin{split}
			1-\f{\al}{2}\psi_\e^{-1}\sum_{k=1}^N x^k\f{\pt\psi_\e}{\pt x^k}=
			\begin{cases}
				\f{1}{2}(2-\al), &\mbox{for } |x|\geq \e, \\
				\psi_\e^{-1}\left(\f{3\e}{8}+\f{3(1-\al)}{4\e}|x|^2-\f{1-2\al}{8\e^3}|x|^4\right), &\mbox{for } |x|\leq \e
			\end{cases}
		\end{split}
	\end{equation*}
	satisfies
	\begin{equation}\label{11.26.4}
		1-\f{\al}{2}\psi_\e^{-1}\sum_{k=1}^N x^k\f{\pt\psi_\e}{\pt x^k}\geq a, \mbox{ with } a=\f{3}{8}\min\{1,2-\al\}
	\end{equation}
	according to \eqref{12.20.1}, \eqref{12.18.1},  and $\e^{-1}\leq \psi_\e^{-1}(x)\leq \f{8}{3}\e^{-1}$ for $x\in B_\e$,  
  as well as Lemma  \ref{12.04.L1} and the fact that $3\al-2\geq 2\al-1$ when $\al\in (1,2)$.
	
Subsequently, we obtain 
	\begin{equation}
		\begin{split}
			\lg D_XH, X\rg_{g_\e}
			&=\llg D_{\sum_{i=1}^N X^i\pxi} \sum_{k=1}^N H^k \pxk, \sum_{j=1}^N X^j\pxj\rrg_{g_\e}\\
			&=\sum_{i,j,k=1}^NX^iX^j\left(\llg\f{\pt H^k}{\pt x^i}\pxk, \pxj\rrg_{g_\e}+H^k\llg D_\pxi \pxk, \pxj\rrg_{g_\e}\right)\\
			&=\sum_{i,j,k=1}^N w_\e^{-1}X^iX^j \left(\de_i^k\de_{kj}+H^k\Ga_{ik}^j\right)\\
			&=\sum_{i=1}^N w_\e^{-1} (X^i)^2-\f{\al}{2}\sum_{i,j,k=1}^Nw_\e^{-1}\psi_\e^{-1} X^iX^j H^k \f{\pt \psi_\e}{\pt x^k}\de_{ij}\\
			&=\left(\sum_{i=1}^N w_\e^{-1}(X^i)^2\right)\left(1-\f{\al}{2}\psi_\e^{-1} \sum_{k=1}^N x^k\f{\pt \psi_\e}{\pt x^k}\right)\geq a |X|_{g_\e}^2
		\end{split}
	\end{equation}
	by $\f{\pt H^k}{\pt x^i}=\de_i^k$, \eqref{11.25.7} and \eqref{11.24.5}, and \eqref{11.24.9} and \eqref{11.26.4}. We complete the proof of this lemma. 
\end{proof}

\begin{notation}\label{12.16.N1}
Let $H, a$ be defined in \eqref{11.25.2} and \eqref{11.26.3} respectively. Set
\begin{equation}\label{11.26.5}
	\begin{split}
		b=\sup_{x\in\Om}|H|_{g_\e}(x),\quad P=\Div_0 H-a.
	\end{split}
\end{equation}
It is obviously that 
\begin{equation}\label{11.27.3}
	P=N-a. 
\end{equation}
 Observe that for all $x\in \Om$,  by \eqref{11.25.7}, we have 
\begin{equation*}
	|H|_{g_\e}^2=\sum_{i=1}^N w_\e^{-1}(H^i)^2=[w_\e(x)]^{-1} |x|^2,
\end{equation*}
and then 
\begin{equation*}
	|H|_{g_\e}^2(x)=
	\begin{cases}
		|x|^{2-\al}, &|x|\geq \e,\\
		\psi_\e^{-\al}(x)|x|^2, &|x|\leq \e,
	\end{cases}
\end{equation*}
by $|x|\leq \psi_\e(x)$  as stated in Lemma \ref{12.17.L1}, we deduce that
\begin{equation}\label{11.26.6}
	R_0^\f{2-\al}{2}<b\leq  M^\f{2-\al}{2}.
\end{equation}
 Let us denote 
\begin{equation}\label{11.26.7}
	c=N^2-a^2.
\end{equation}
 From \eqref{11.26.3}, it follows that 
 $0<a<1$, and thus $c\in (0,+\iy)$.    
 
 It is important to note that the constants  $a, b, c$ and $P$ 
 are independent of $\e\in (0,R_0)$ by \eqref{11.26.3}, and \eqref{11.27.3}, and \eqref{11.26.6} and \eqref{11.26.7}. 
\end{notation}

\begin{lemma}\label{11.26.L2}
	 Let $H$  be given by \eqref{11.25.2}, and $b$ 
 be defined as in \eqref{11.26.5}. Suppose
$\vp$
 is a solution to \eqref{11.14.2}. Then, we have
	\begin{equation}\label{11.27.1}
		\left|\left( \pt_t\vp,H(\vp)+\f{1}{2}P\vp\right)_{L^2(\Om)}\right|\leq bE_\e(0),
	\end{equation}
	where
	\begin{equation}\label{11.27.2}
		E_\e(t)=\f{1}{2}\int_\Om \left((\pt_t\vp)^2+|\nabla_{g_\e}\vp|_{g_\e}^2\right)\df x=E_\e(0)=\f{1}{2}\left(\|\vp^1\|_{L^2(\Om)}^2+\left\||\nabla_{g_\e}\vp^0|_{g_\e}\right\|_{L^2(\Om)}^2\right).
	\end{equation}
\end{lemma}

\begin{proof}
Applying the divergence theorem and utilizing the fact that
$\vp=0$ on $\pt\Om$, we obtain
	\begin{equation*}
		\int_\Om H(\vp^2)\df x=\int_\Om \sum_{i=1}^N x^i\f{\pt \vp^2}{\pt x^i}\df x=-N\int_\Om \vp^2\df x,
	\end{equation*}
	 and consequently,
	\begin{equation*}
		\begin{split}
			\left\|H(\vp)+\f{1}{2}P\vp\right\|_{L^2(\Om)}^2
			&=\|H(\vp)\|_{L^2(\Om)}^2+(H(\vp), P\vp)_{L^2(\Om)}+\f{1}{4}\|P\vp\|_{L^2(\Om)}^2\\
			&=\|H(\vp)\|_{L^2(\Om)}^2+\f{1}{2}\int_\Om PH(\vp^2)\df x+\f{1}{4}\|P\vp\|_{L^2(\Om)}^2\\
			&=\|H(\vp)\|_{L^2(\Om)}^2-\f{c}{4}\int_\Om \vp^2\df x.
		\end{split}
	\end{equation*}
	 by \eqref{11.26.5}, \eqref{11.27.3}, and \eqref{11.26.7}. Here, $P$
 and $c$  are defined in \eqref{11.27.3} and \eqref{11.26.7}, respectively. Therefore, from \eqref{11.25.7} and \eqref{11.26.5}, and considering
	\begin{equation*}
		\begin{split}
			\|H(\vp)\|_{L^2(\Om)}^2
			&=\int_\Om \left(\sum_{i=1}^N H^i\f{\pt\vp}{\pt x^i}\right)^2\df x\leq \int_\Om \left(\sum_{i=1}^N w_\e^{-1}(H^i)^2\right)\left(\sum_{i=1}^N w_\e \left(\f{\pt\vp}{\pt x^i}\right)^2\right)\df x\\
			&=\int_\Om |H(x)|_{g_\e}^2\left[\sum_{i=1}^N w_\e \left(\f{\pt\vp}{\pt x^i}\right)^2\right]\df x\leq b^2\left\||\nabla_{g_\e}\vp|_{g_\e}\right\|_{L^2(\Om)}^2,
		\end{split}
	\end{equation*}
	where $b$ is defined by \eqref{11.26.5}, we  can deduce that
	\begin{equation*}
		\begin{split}
			\left|\left(\pt_t\vp, H(\vp)+\f{1}{2}P\vp\right)_{L^2(\Om)}\right|
			&\leq \|\pt_t\vp\|_{L^2(\Om)}\left\|H(\vp)+\f{1}{2}P\vp\right\|_{L^2(\Om)}\leq \|\pt_t\vp\|_{L^2(\Om)}\|H(\vp)\|_{L^2(\Om)}\\
			&\leq b\|\pt_t\vp\|_{L^2(\Om)}\left\||\nabla_{g_\e}\vp|_{g_\e}\right\|_{L^2(\Om)}\leq bE_\e(0).
		\end{split}
	\end{equation*}
	 This completes the proof of the lemma.
\end{proof}

\begin{theorem}\label{11.27.T1}
	Let $a$ and $b$ be defined as in Notation \ref{12.16.N1}. Define 
	\begin{equation}\label{11.27.4}
		\theta=\sup_{x\in\pt\Om}\f{H\cdot\nu}{|\nu_\e|_{g_\e}^2}.
	\end{equation}
	Then,  for all $T>0$ and  all  $\vp^0\in H_0^1(\Om), \vp^1\in L^2(\Om)$,
	\begin{equation}\label{11.28.8}
		\int_0^T\int_{\Ga_0} \left(\f{\pt \vp}{\pt \nu_\e}\right)^2\df S\df t\geq \f{2}{\theta}\left(aT-2b\right)E_\e(0),
	\end{equation}
	where $\vp$   is the solution of \eqref{11.14.2} corresponding to  $(\vp^0,\vp^1, 0)$, and
	\begin{equation}\label{11.28.9}
		\Ga_0=\left\{x\in\pt\Om\colon H(x)\cdot \nu(x)>0\right\}.
	\end{equation}
\end{theorem}

\begin{proof}
	 First, observe that from \eqref{11.25.7} and  $\nu_\e=w_\e \nu$,  we have  $|\nu_\e|_{g_\e}^2(x)=\sum_{i=1}^Nw_\e(\nu^i)^2(x)=|x|^\al$ on $\pt\Om$. Moreover, given   $H\cdot\nu=\sum_{i=1}^N x^i \nu^i$ and $R_0\leq |x|\leq M$ for all $x\in \pt\Om$, we can conclude that
	\begin{equation}\label{11.27.6}
		\theta=\sup_{x\in\pt\Om}\f{H\cdot\nu}{|\nu_\e|_{g_\e}^2}=\sup_{x\in\pt\Om}\f{H\cdot\nu}{|x|^\al}\in (0,+\iy) \mbox{ and $\theta$ is independent of $\e>0$}.
	\end{equation}
	
	Let $x\in \pt\Om$.  Then we can express
	\begin{equation}\label{11.28.3}
		\nabla_{g_\e}\vp(x)=\llg \nabla_{g_\e}\vp(x), \f{\nu_\e(x)}{|\nu_\e|_{g_\e}}\rrg_{g_\e}\f{\nu_\e(x)}{|\nu_\e|_{g_\e}}+Y(x) \mbox{ with } \lg Y(x), \nu_\e(x)\rg_{g_\e}=0.
	\end{equation}
	 Note that from \eqref{11.28.2} and \eqref{11.28.1}, we obtain
	\begin{equation*}
		Y(x)\cdot \nu(x)=\llg Y(x), \nu_\e(x)\rrg_{g_\e}=0,
	\end{equation*}
	 which implies that
 $Y(x)\in (\pt\Om)_x$, the tangent space of $\pt\Om$ at $x$.   Combining this with \eqref{11.28.3} and $\vp=0$ on $\pt\Om$, we get, for each $x\in\pt\Om$,
	\begin{equation}\label{11.28.4}
		\begin{split}
			\left|\nabla_{g_\e}\vp\right|_{g_\e}^2
			&=\nabla_{g_\e}\vp(\vp)=\f{1}{|\nu_\e(x)|_{g_\e}^2}\llg \nabla_{g_\e}\vp(x), \nu_\e(x)\rrg_{g_\e}^2+Y(\vp)=\f{1}{|\nu_\e|_{g_\e}^2}\left(\f{\pt\vp}{\pt\nu_\e}\right)^2.
		\end{split}
	\end{equation}
	Similar to \eqref{11.28.3}, $H$ can be written as
	\begin{equation}\label{11.28.5}
		\begin{split}
			H=\llg H(x), \f{\nu_\e(x)}{|\nu_\e(x)|_{g_\e}}\rrg_{g_\e}\f{\nu_\e(x)}{|\nu_\e(x)|_{g_\e}}+Z(x),
		\end{split}
	\end{equation}
	where $Z(x)\in (\pt\Om)_x$.  Combining this with \eqref{11.28.2}, \eqref{11.28.1}, and $\vp=0$ on $\pt\Om$,  we 
have, for each $x\in\pt\Om$,
	\begin{equation}\label{11.28.6}
		\begin{split}
			H(\vp)=\f{\llg H(x), \nu_\e(x)\rrg_{g_\e}}{|\nu_\e(x)|_{g_\e}^2}\f{\pt\vp}{\pt \nu_\e}=\f{H(x)\cdot \nu(x)}{|\nu_\e(x)|_{g_\e}^2}\f{\pt \vp}{\pt \nu_\e}.
		\end{split}
	\end{equation}
	
	Now, we have
	\begin{eqnarray*}
			&&\f{\theta}{2}\int_0^T\int_{\Ga_0}\left(\f{\pt\vp}{\pt\nu_\e}\right)^2\df S\df t \crr
			&&\geq  \f{1}{2}\iint_{\pt Q} \left(\f{\pt\vp}{\pt \nu_\e}\right)^2\f{H\cdot \nu}{|\nu_\e|_{g_\e}^2}\df S\df t\crr
			&&=\iint_{\pt Q}\f{\pt\vp}{\pt\nu_\e}H(\vp)\df S\df t+\f{1}{2}\iint_{\pt Q}\left((\pt_t\vp)^2-|\nabla_{g_\e}\vp|_{g_\e}^2\right)H\cdot \nu\df S\df t\crr
			&&=\int_\Om [(\pt_t\vp)H(\vp)](t)\df x\bigg|_{t=0}^{t=T}\\
			&&\hspace{4.5mm}+\iint_Q \llg D_{\nabla_{g_\e}\vp } H, \nabla_{g_\e}\vp\rrg_{g_\e}\df x\df t+\f{1}{2}\iint_Q \left((\pt_t\vp)^2-|\nabla_{g_\e}\vp|_{g_\e}^2\right)\Div_0(H)\df x\df t\crr
			&&= \f{a}{2}\iint_Q \left((\pt_t\vp)^2+|\nabla_{g_\e}\vp|_{g_\e}^2\right)\df x\df t\crr
			&&\hspace{4.5mm}+\iint_Q \llg D_{\nabla_{g_\e}\vp } H, \nabla_{g_\e}\vp\rrg_{g_\e}\df x\df t-a\iint_Q |\nabla_{g_\e}\vp|_{g_\e}^2\df x\df t\crr
			&&\hspace{4.5mm}+\int_\Om [(\pt_t\vp)H(\vp)](t)\df x\bigg|_{t=0}^{t=T}+\f{1}{2}\iint_Q \left((\pt_t\vp)^2-|\nabla_{g_\e}\vp|_{g_\e}^2\right)P\df x\df t,\nonumber 
	\end{eqnarray*}
	 where for the first inequality, we employed \eqref{11.27.4}; for the first equality, we utilized \eqref{11.28.4}, \eqref{11.28.6}, along with the fact that $\pt_t\vp=0$ on $\pt\Om$;   for the second equality, we apply 1) in Proposition \ref{11.24.P1}; and for the last equality, we made use of $\Div_0(H)=N$ and \eqref{11.27.3}.

Subsequently, drawing upon 2) in Proposition \ref{11.24.P1}, as well as \eqref{11.26.3}, \eqref{11.26.5}, \eqref{11.27.3}, and Lemma \ref{11.26.L2}, we arrive at the following result
	\begin{equation*}
		\begin{split}
			\f{\theta}{2}\int_0^T\int_{\Ga_0}\left(\f{\pt\vp}{\pt\nu_\e}\right)^2\df S\df t
			&\geq a TE_\e(0)+\int_\Om (\pt_t\vp)\left(H(\vp)+\f{1}{2}P\vp\right)\df x\bigg|_{t=0}^{t=T}\geq (aT-2b)E_\e(0).
		\end{split}
	\end{equation*}
	 This completes the proof of the theorem.
\end{proof}

\begin{corollary}
	 Under the conditions specified in Theorem \ref{11.27.T1}, the following inequality holds
	\begin{equation*}
		\int_0^T\int_{\Ga_0}\left(\f{\pt \vp}{\pt \nu}\right)^2\df S\df t\geq \f{2}{\theta M^{2\al}} (aT-2b)E_\e(0).
	\end{equation*}
	 Furthermore, when $T>\f{2b}{a}$, 
	\begin{equation}\label{12.04.2}
		\int_0^T\int_{\Ga_0}\left(\f{\pt \vp}{\pt \nu}\right)^2\df S\df t\geq CE_\e(0),
	\end{equation}
	where the constant $C>0$ depends only on $\al, M, T$ and $\Om$.
\end{corollary}

\begin{proof}
	 Given that  $\nu_\e=w_\e \nu$   and leveraging Theorem \ref{11.27.T1}, we can derive the desired result.
\end{proof}

\subsection{Approximation}

 With the aforementioned preparations completed, we now proceed to establish one of the principal results of this paper.

\begin{lemma}\label{12.04.L2}
	Let $a,b$ be defined by Notation \ref{12.16.N1}. Suppose 
 $\vp^0\in C_0^\iy(\Om), \vp^1\in H_0^1(\Om)$, and $f=0$. Then, for $T>\f{2b}{a}$,  there exist two  positive constants $C_1$ and $C_2$,  
  which depend only on $\al, M, T$ and $\Om$, such that
	\begin{equation}\label{12.04.1}
		\begin{split}
			C_1E(0)\leq \int_0^T\int_{\Ga_0}\left(\f{\pt\vp}{\pt\nu}\right)^2\df S\df t\leq C_2E(0),
		\end{split}
	\end{equation}
	where $\vp$  represents the weak solution of \eqref{12.14.1} corresponding to $(\vp^0,\vp^1, f=0)$, and
	\begin{equation*}
		E(0)=\f{1}{2}\int_\Om \left((\vp^1)^2+w\nabla\vp^0\cdot \nabla\vp^0\right)\df x.
	\end{equation*}
\end{lemma}

\begin{proof}
	 The second inequality in \eqref{12.04.1} is a direct consequence of \eqref{11.18.1} in Theorem \ref{11.14.T1}.

On the one hand, invoking Lemmas \ref{12.01.L1} and \ref{02.05.L1}, we deduce that $\vp^0\in D(\mcA)\cap D(\mcA_\e)\ (0<\e<R_0)$, and
 and consequently, \eqref{11.21.7} holds. On the other hand, as as $\e\ra 0$, we have   $E_\e(0)\ra E(0)$ as $\e\ra 0$,
 which can be shown by
	\begin{equation*}
		\begin{split}
			\int_\Om w_\e|\nabla\vp^0|^2\df x-\int_\Om w|\nabla \vp^0|^2\df x=\int_{B_\e}(w_\e-w)|\nabla\vp^0|^2\df x\leq 2\e^\al \int_{B_\e}|\nabla \vp^0|^2\df x\ra 0. 
		\end{split}
	\end{equation*}
	 Therefore, by taking the limit as  $\e\ra 0$ (by extracting a subsequence), we obtain
	\begin{equation*}
		\int_0^T\int_{\Ga_0}\left(\f{\pt\vp}{\pt\nu}\right)^2\df S\geq C_1E(0).
	\end{equation*}
	 This completes the proof of this lemma.
\end{proof}

 The following theorem constitutes a principal result of this paper.

\begin{theorem}\label{12.04.T3}
	Let $a,b$ be defined by Notation \ref{12.16.N1}.  Suppose  $\vp^0\in H_0^1(\Om;w), \vp^1\in L^2(\Om)$ and $f=0$. Then, for $T>\f{2b}{a}$, there exist two positive constants $C_1$ and $C_2$,  and both depend only on
 $\al, M, T$ and $\Om$,  such that
	\begin{equation}\label{12.04.3}
		\begin{split}
			C_1E(0)\leq \int_0^T\int_{\Ga_0}\left(\f{\pt\vp}{\pt\nu}\right)^2\df S\df t\leq C_2E(0),
		\end{split}
	\end{equation}
	where $\vp$  represents the weak solution of \eqref{12.14.1} corresponding to $(\vp^0,\vp^1,f=0)$.
\end{theorem}

\begin{proof}
	 The second inequality in \eqref{12.04.1} is a direct consequence of \eqref{11.18.1} in Theorem \ref{11.14.T1}.
	
Let $\{\vp_\e^0\}_{0<\e<R_0}\subset C_0^\iy(\Om)$ such that $\vp_\e^0\to \vp^0$ in $H_0^1(\Om;w)$, and $\{\vp_\e^1\}_{0<\e<R_0}\subset C_0^\iy(\Om)$ such that $\vp_\e^1\to \vp^1$ in $L^2(\Om)$. Consider the following equation:
\begin{equation*}
	\begin{cases}
		\partial_{tt}\eta_\e-\mcA \eta_\e=0, &\text{in }Q, \\
		\eta_\e=0, &\text{on }\partial Q, \\
		\eta_\e(0)=\vp_\e^0, \partial_t\eta_\e(0)=\vp_\e^1, &\text{in }\Om,
	\end{cases}
\end{equation*}
Then, from \eqref{11.18.1} in Theorem \ref{11.14.T1}, we have
\begin{equation*}
	\left\|\frac{\partial (\eta_\e-\vp)}{\partial \nu}\right\|_{L^2(\partial Q)}\leq C\left(\|\vp^0-\vp_\e^0\|_{H_0^1(\Om;w)}+\|\vp^1-\vp_\e^1\|_{L^2(\Om)}\right),
\end{equation*}
that is, $\frac{\partial\eta_\e}{\partial \nu}\to \frac{\partial\vp}{\partial \nu}$ strongly in $L^2(\partial Q)$ as $\vp_\e^0\to \vp^0$ in $H_0^1(\Om;w)$ and $\vp_\e^1\to \vp^1$ in $L^2(\Om)$,
where the constant $C>0$ depends only on $\al$, $M$, $T$, and $\Om$. Note that by Lemma \ref{12.04.L2}, we have
\begin{equation*}
	\int_0^T\int_{\Ga_0}\left(\frac{\partial \eta_\e}{\partial\nu}\right)^2\df S\df t\geq C_1\int_\Om \left((\vp_\e^1)^2+w|\nabla \vp_\e^0|^2\right)\df x
\end{equation*}
 where the constant $C_1>0$ depends only on $\al$, $M$, $T$, and $\Om$. Hence
\begin{equation*}
	\int_0^T\int_{\Ga_0}\left(\frac{\partial\vp}{\partial\nu}\right)^2\df S\df t\geq C_1\int_\Om \left((\vp^1)^2+w|\nabla\vp^0|^2\right)\df x=C_1E(0).
\end{equation*}
This concludes the proof of this theorem.
\end{proof}

\subsection{Exactly null controllability}

Let $\Ga_0$  be defined as in \eqref{11.28.9}. We say that equation \eqref{12.16.2} is exactly null controllable at time
$T>0$  if there exists a control function $u\in L^2(\Sigma_0)$  such that the solution
$\vp$  of the following equation
\begin{equation}\label{12.16.2}
	\begin{cases}
		\pt_{tt}\vp+\mcA\vp=0, &\mbox{in }Q, \\
		\vp=u, &\mbox{on }\Sigma_0, \\
		\vp=0, &\mbox{on }\Sigma-\Sigma_0, \\
		\vp(0)=\vp^0,  \pt_t\vp(0)=\vp^1, &\mbox{in }\Om
	\end{cases}
\end{equation}
satisfies $\vp(\cdot, T)=\pt_t\vp(\cdot,T)=0$, where $\vp^0\in L^2(\Om)$ and $\vp^1\in H^{-1}(\Om;w)$.
Here $\Sigma_0=\Ga_0\ts (0,T)$ and $\Sigma=\pt\Om\ts (0,T)$, and  $\vp^0\in H_0^1(\Om;w)$ and $\vp^1\in L^2(\Om)$, which are mentioned in 
\dref{12.16.2*}.

Here, $\Sigma_0=\Ga_0\ts (0,T)$ and $\Sigma=\pt\Om\ts (0,T)$.  Additionally,
 $\vp^0\in H_0^1(\Om;w)$ and $\vp^1\in L^2(\Om)$, as specified in \dref{12.16.2*}.

 The well-posedness of equation \eqref{12.16.2} has been thoroughly investigated in \cite[Definition 2.15 and Lemma 2.16]{Yang2}. We now proceed to prove Theorem \ref{12.04.T4}.

\vspace{0.2cm} 

\noindent {\bf {Proof of Theorem \ref{12.04.T4}}.}
	 This result follows as a standard consequence of J.-L. Lions' Hilbert Uniqueness Method, as outlined in \cite[Theorem 1.1]{Lions}. According to this method, the controllability of equation \eqref{12.16.2} is equivalent to the validity of the following so-called observability inequality:
	\begin{equation*}
		E(0)\leq C\int_0^T\int_{\Ga_0}\left(\f{\pt\vp}{\pt \nu}\right)^2\df S\df t,
	\end{equation*}
	where $C>0$  represents a positive constant. For further reference, similar results can be found in \cite[Theorem 1.3]{Yao}, \cite[Theorem 2.2, Chapter 2.3, p. 15]{Zuazua}, and \cite[Theorem 2.42, Chapter 2.3.2, p. 56]{Coron}.
\hfill $\Box$ 

\section{Concluding Remarks}\label{Se5}

{{In this paper, we first approximate the degenerate hyperbolic equation by uniformly hyperbolic equations and establish the controllability (i.e., the observability inequality) for the degenerate hyperbolic equation with a degenerate interior point. To the best of our knowledge, these results do not appear in the existing literature. These findings imply that the degenerate hyperbolic equation can be controlled at the boundary.

As is well known, a natural approach to address the controllability of degenerate hyperbolic equations is to approximate them by uniformly hyperbolic equations. If the degenerate points are interior points, we approximate the coefficients in the partial differential operator $\mathcal{A}$ using appropriate functions. If the degenerate points lie on the boundary, the shape design method must be employed. However, we note that the behavior of wave propagation near the degenerate points remains an {\it open problem}. This is a purely mathematical issue.

We emphasize that the cut-off method plays a fundamental role in the analysis of degenerate partial differential equations. The so-called cut-off method involves isolating problematic regions (the “bad” points or regions) and focusing on the well-behaved ones (the “good” points or regions), allowing us to derive useful properties on the latter. In the proof of Theorem \ref{11.14.T1}, this method is essential for obtaining the improved regularity estimates \eqref{02.04.8} and \eqref{02.13.3}, as well as for handling the normal vector $\bs{n}$ in the hidden regularity estimate. Furthermore, even in the construction of solutions to \eqref{12.16.2}, the Dirichlet map relies on the cut-off method (see \cite[Lemmas 2.10 and 2.11]{Yang2}). In degenerate partial differential equations, the “bad” regions typically correspond to points where regularity cannot be improved. At such points, certain quantities may fail to be well defined or controllable—for instance, the normal derivative $\frac{\partial \varphi}{\partial \nu}$ may not exist on the degenerate part of the boundary. Consequently, one cannot directly obtain enhanced regularity or boundary estimates there. This difficulty motivates us to isolate these problematic regions and instead perform the analysis on the “good” regions, where stronger regularity and controllability properties can be established.}}

\end{document}